\documentclass[12pt]{amsproc}

\usepackage{newclude}

\usepackage{amssymb}
\usepackage{amsfonts}
\usepackage[latin1]{inputenc}
\usepackage[mathscr]{euscript}
\usepackage{enumerate}
\usepackage[all]{xy}
\usepackage{tkz-graph}
\usetikzlibrary{arrows}
\usetikzlibrary{positioning}

\usepackage[
margin=1in,
includefoot,
footskip=30pt,
]{geometry} 

\usepackage{mathtools}
\usepackage{color}

\newtheorem{theorem}[equation]{Theorem}
\newtheorem{lemma}[equation]{Lemma}
\newtheorem{proposition}[equation]{Proposition}
\newtheorem{corollary}[equation]{Corollary}

\theoremstyle{definition}
\newtheorem{definition}[equation]{Definition}

\theoremstyle{remark}
\newtheorem{remark}[equation]{Remark}

\numberwithin{equation}{section}

\newcommand{\dmal}[1]{\begin{align*}{#1}\end{align*}}              
\newcommand{\newf}[3]{{#1}:{#2}\longrightarrow {#3}}					 
\newcommand{\newfd}[5]{\begin{array}{cccc}{#1}: &{#2} &\longrightarrow &{#3} \\& {#4} &\longmapsto &{#5}\end{array}}

\newcommand{\la}[1]{\text{$\mathcal{#1}$}}
\newcommand{\lb}[1]{\text{$\mathscr{#1}$}}

\newcommand{\powerset}[1]{\text{$\lb{P}(#1)$}}								
\newcommand{\eword}{\text{$\omega$}}											

\newcommand{\xia}{\text{$\xi^\alpha$}}											

\newcommand{\usetr}[2]{\text{$\uparrow_{\scriptscriptstyle{#2}}\hspace{-0.1cm}{#1}$}} 

\newcommand{\dgraphg}[1]{\text{$\lb{#1}$}}									
\newcommand{\dgraphupleg}[3]{\text{$\dgraphg{#1}=(\dgraphg{#1}^0,\dgraphg{#1}^1,#2,#3)$}}	

\newcommand{\dgraph}{\text{$\dgraphg{E}$}}									
\newcommand{\dgraphuple}{\text{$\dgraphupleg{E}{r}{s}$}}					
\newcommand{\alfg}[1]{\text{$\lb{#1}$}}										
\newcommand{\acfg}[1]{\text{$\lb{#1}$}}										
\newcommand{\lbfg}[1]{\text{$\lb{#1}$}}										
\newcommand{\acfrg}[1]{\text{$\acfg{B}_{#1}$}}								

\newcommand{\alf}{\text{$\alfg{A}$}}											
\newcommand{\acf}{\text{$\acfg{B}$}}											
\newcommand{\lbf}{\text{$\lbfg{L}$}}											
\newcommand{\acfra}{\text{$\acfg{B}_{\alpha}$}}								
\newcommand{\lgraphg}[2]{\text{$(\dgraphg{#1},\lbfg{#2})$}}				
\newcommand{\lspaceg}[3]{\text{$(\dgraphg{#1},\lbfg{#2},\acfg{#3})$}}

\newcommand{\lgraph}{\text{$\lgraphg{E}{L}$}}								
\newcommand{\lspace}{\text{$\lspaceg{E}{L}{B}$}}							

\newcommand{\awsetg}[2]{\text{$\lbfg{#1}^{#2}$}}	
\newcommand{\awplus}{\text{$\awsetg{L}{\scriptscriptstyle{\geq 1}}$}}								
\newcommand{\awn}[1]{\text{$\awsetg{L}{#1}$}}								
\newcommand{\awstar}{\text{$\awsetg{L}{\ast}$}}							
\newcommand{\awinf}{\text{$\awsetg{L}{\infty}$}}							
\newcommand{\awleinf}{\text{$\awsetg{L}{\scriptscriptstyle{\leq\infty}}$}}					

\newcommand{\fset}[1]{\text{$\textsf{#1}$}}										
\newcommand{\filt}{\text{$\fset{F}$}}											

\newcommand{\ftight}{\text{$\textsf{T}$}}										
\newcommand{\ftightw}[1]{\text{$\textsf{T}_{#1}$}}

\newcommand{\ftg}[1]{\text{$\la{#1}$}}											
\newcommand{\ft}{\text{$\ftg{F}$}}												

\newcommand{\card}[1]{\# #1}

\newcommand{\scj}{\text{$\subseteq$}}

\newcommand{\nn}{\text{$\mathbb{N}$}}

\allowdisplaybreaks

\begin{document}

\title[Groupoid models for labelled spaces]{Groupoid models for the C*-algebra of labelled spaces}

\author[G. Boava \and G. de Castro \and F. Mortari]{Giuliano Boava \and Gilles G. de Castro \and Fernando de L. Mortari}
\address{Departamento de Matemática, Universidade Federal de Santa Catarina, 88040-970 Florianópolis SC, Brazil.}
\email{g.boava@ufsc.br \\ gilles.castro@ufsc.br \\ \newline fernando.mortari@ufsc.br}


\begin{abstract}
We define a groupoid from a labelled space and show that it is isomorphic to the tight groupoid arising from an inverse semigroup associated with the labelled space. We then define a local homeomorphism on the tight spectrum that is a generalization of the shift map for graphs, and show that the defined groupoid is isomorphic to the Renault-Deaconu groupoid for this local homeomorphism. Finally, we show that the C*-algebra of this groupoid is isomorphic to the C*-algebra of the labelled space as introduced by Bates and Pask.
\end{abstract}

\maketitle

\section{Introduction}

The definition of a C*-algebra associated with a labelled graph, or more precisely a labelled \emph{space}, was first given by Bates and Pask in \cite{MR2304922}. Among the examples given by them were the classes of graph C*-algebras \cite{MR1432596,MR2135030}, ultragraph C*-algebras \cite{MR2050134} and Carlsen-Matsumoto algebras \cite{MR2091486,MR2380472}. These encompass other important classes of C*-algebras such as Cuntz algebras \cite{MR0467330}, Cuntz-Krieger algebras \cite{MR561974} and Exel-Laca algebras \cite{MR1703078}. The original definition was later revised in \cite{MR2542653} and then in \cite{MR3614028}. We point out that the authors in \cite{MR3680957} independently found the revised definition given in \cite{MR3614028}.

One powerful tool when studying C*-algebras is to describe them as groupoid C*-algebras as pioneered by Renault \cite{MR584266}. Several of the classes mentioned above were shown to be isomorphic to a groupoid C*-algebra \cite{MR584266,MR1432596,MR1770333,MR1962477,MR2457327,MR2680816}.

The main goal of this paper is to prove that the C*-algebra associated with a labelled space is also isomorphic to a groupoid C*-algebra. This was partially done for the countable case in a paper by Carlsen, Ortega and Pardo \cite[Theorem 8.3 and Example 11.1]{MR3606190} where they work in the setting of Boolean dynamical systems, under an additional assumption on the labelled space. However, certain C*-algebras, such as those associated with shift spaces, when seen from the point of view of labelled graphs, deal with uncountable sets (in the case of the shift, the vertices are the points of the shift). Their approach was to use a universal property of groupoid C*-algebras described by Exel in \cite{MR2419901} in the case of second countable étale groupoids. In a footnote of \cite{MR2419901}, Exel observes that it might be possible to circumvent the second countability hypothesis by working out the results therein using Baire $\sigma$-algebras instead of Borel $\sigma$-algebras, and it is thus possible that Carlsen, Ortega and Pardo's results could be generalized to the uncountable case. We follow a different path, more in the lines of what was done by Kumjian, Pask, Raeburn and Renault \cite{MR1432596}.

Our approach is also heavily based in the framework developed by Exel in \cite{MR2419901}, working with inverse semigroups, their semilattices of idempotents  and tight spectra. The interplay between inverse semigroups, groupoids and C*-algebras already appears in Renault's monograph \cite{MR584266} and in Paterson's book \cite{MR1724106}. Paterson associated a universal groupoid with an inverse semigroup, however in several cases this was not the ``correct'' groupoid, in a sense; it was also necessary to restrict this groupoid by looking at the object used to define the inverse semigroup.  One of Exel's main motivations was to find the correct groupoid solely with the inverse semigroup.

The original purpose of \cite{MR3648984} and \cite{MR3680957} was to lay down the framework for the results now presented here. In \cite{MR3648984}, the authors defined an inverse semigroup associated with a labelled space by mimicking the product of elements inside the C*-algebras associated with this labelled space; contrary to what is done by Carlsen, Ortega and Pardo in \cite{MR3606190}, however, the elements of the inverse semigroup are not seen inside the C*-algebra and, as it turns out, they need not even be the same \cite[Example~7.3]{MR3680957}. We then gave a description of the tight spectrum associated with the defined inverse semigroup that is somewhat similar to the \emph{boundary path space} of a graph \cite{MR3119197} and, depending on the labelled space, they actually coincide \cite[Proposition~6.9]{MR3648984}.

The boundary path space of a graph is the unit space of the groupoid that gives the graph C*-algebra; it is also the spectrum of a commutative C*-subalgebra, called the \emph{diagonal} subalgebra. The authors extended this result to the case of labelled spaces \cite{MR3680957} and, in doing so, the revised definition for labelled space C*-algebras mentioned in the first paragraph was found.

After reviewing the necessary tools, we begin our paper by describing, from a labelled space, a groupoid isomorphic to the tight groupoid defined by Exel \cite{MR2419901}, but concretely instead of as a quotient. First we work in a purely algebraic setting and then in a topological setting. By using some \emph{cutting} and \emph{gluing} maps defined in \cite{MR3680957}, we build a local homeomorphism and show that our groupoid can be seen as a Renault-Deaconu groupoid \cite{MR584266,MR1233967,MR1770333}. We then prove that the C*-algebra of this groupoid is isomorphic to the C*-algebra associated with the labelled space.

\section{Preliminaries}

\subsection{Labelled spaces}\label{subsect.labelled.spaces}

A \emph{(directed) graph} $\dgraphuple$ consists of non-empty sets $\dgraph^0$ (of \emph{vertices}), $\dgraph^1$ (of \emph{edges}), and \emph{range} and \emph{source} functions $r,s:\dgraph^1\to \dgraph^0$. A vertex $v$ such that $s^{-1}(v)=\emptyset$ is called a \emph{sink}, and the set of all sinks is denoted by $\dgraph^0_{sink}$.

A \emph{path of length $n$} on a graph $\dgraph$ is a sequence $\lambda=\lambda_1\lambda_2\ldots\lambda_n$ of edges such that $r(\lambda_i)=s(\lambda_{i+1})$ for all $i=1,\ldots,n-1$. We write $|\lambda|=n$ for the length of $\lambda$ and regard vertices as paths of length $0$. $\dgraph^n$ stands for the set of all paths of length $n$ and $\dgraph^{\ast}=\cup_{n\geq 0}\dgraph^n$. We extend the range and source maps to $\dgraph^{\ast}$ by defining $s(\lambda)=s(\lambda_1)$ and $r(\lambda)=r(\lambda_n)$ if $n\geq 2$ and $s(v)=v=r(v)$ for $n=0$. Similarly, we define a \emph{path of infinite length} (or an \emph{infinite path}) as an infinite sequence $\lambda=\lambda_1\lambda_2\ldots$ of edges such that $r(\lambda_i)=s(\lambda_{i+1})$ for all $i\geq 1$; for such a path, we write $|\lambda|=\infty$; $\dgraph^{\infty}$ denotes the set of all infinite paths.

A \emph{labelled graph} consists of a graph $\dgraph$ together with a surjective \emph{labelling map} $\lbf:\dgraph^1\to\alf$, where $\alf$ is a fixed non-empty set, called an \emph{alphabet}, and whose elements are called \emph{letters}. $\alf^{\ast}$ stands for the set of all finite \emph{words} over $\alf$, together with the \emph{empty word} \eword, and $\alf^{\infty}$ is the set of all infinite words over $\alf$. 	A labelled graph is said to be \emph{left-resolving} if for each $v\in\dgraph^0$ the restriction of $\lbf$ to $r^{-1}(v)$ is injective.

The labelling map $\lbf$ extends in the obvious way to $\lbf:\dgraph^n\to\alf^{\ast}$ and $\lbf:\dgraph^{\infty}\to\alf^{\infty}$. $\awn{n}=\lbf(\dgraph^n)$ is the set of \emph{labelled paths $\alpha$ of length $|\alpha|=n$}, and $\awinf=\lbf(\dgraph^{\infty})$ is the set of \emph{infinite labelled paths}. We consider $\eword$ as a labelled path with $|\eword|=0$, and set $\awplus=\cup_{n\geq 1}\awn{n}$, $\awstar=\{\eword\}\cup\awplus$, and $\awleinf=\awstar\cup\awinf$.

For $\alpha\in\awstar$ and $A\in\powerset{\dgraph^0}$ (the power set of $\dgraph^0$), the \emph{relative range of $\alpha$ with respect to} $A$ is the set
\[r(A,\alpha)=\{r(\lambda)\ |\ \lambda\in\dgraph^{\ast},\ \lbf(\lambda)=\alpha,\ s(\lambda)\in A\}\]
if $\alpha\in\awplus$ and $r(A,\eword)=A$ if $\alpha=\eword$. The \emph{range of $\alpha$}, denoted by $r(\alpha)$, is the set \[r(\alpha)=r(\dgraph^0,\alpha),\]
so that $r(\eword)=\dgraph^0$ and, if $\alpha\in\awplus$, then  $r(\alpha)=\{r(\lambda)\in\dgraph^0\ |\ \lbf(\lambda)=\alpha\}$.

We also define \[\lbf(A\dgraph^1)=\{\lbf(e)\ |\ e\in\dgraph^1\ \mbox{and}\ s(e)\in A\}=\{a\in\alf\ |\ r(A,a)\neq\emptyset\}.\]

A labelled path $\alpha$ is a \emph{beginning} of a labelled path $\beta$ if $\beta=\alpha\beta'$ for some labelled path $\beta'$; also, $\alpha$ and $\beta$ are \emph{comparable} if either one is a beginning of the other. If $1\leq i\leq j\leq |\alpha|$, let $\alpha_{i,j}=\alpha_i\alpha_{i+1}\ldots\alpha_{j}$ if $j<\infty$ and $\alpha_{i,j}=\alpha_i\alpha_{i+1}\ldots$ if $j=\infty$. If $j<i$ set $\alpha_{i,j}=\eword$.

A \emph{labelled space} is a triple $\lspace$ where $\lgraph$ is a labelled graph and $\acf$ is a family of subsets of $\dgraph^0$ that is closed under finite intersections and finite unions, contains all $r(\alpha)$ for $\alpha\in\awplus$, and is \emph{closed under relative ranges}, that is, $r(A,\alpha)\in\acf$ for all $A\in\acf$ and all $\alpha\in\awstar$. If in addition $\acf$ is closed under relative complements, we say the labelled space is \emph{normal}. Finally, a labelled space $\lspace$ is \emph{weakly left-resolving} if for all $A,B\in\acf$ and all $\alpha\in\awplus$ we have $r(A\cap B,\alpha)=r(A,\alpha)\cap r(B,\alpha)$.

For a given $\alpha\in\awstar$, let \[\acfra=\acf\cap\powerset{r(\alpha)}.\] If the labelled space is weakly left-resolving and normal, then the set $\acfra$ is a Boolean algebra for each $\alpha\in\awplus$, and  $\acfrg{\eword}=\acf$ is a generalized Boolean algebra as in \cite{MR1507106}. By Stone duality  there is a topological space associated with each $\acfra$ with $\alpha\in\awstar$, which we denote by $X_{\alpha}$, consisting of the set of ultrafilters in $\acfra$. 

\subsection{The inverse semigroup of a labelled space}\label{subsection:inverse.semigroup}

For a given labelled space $\lspace$ \emph{that is weakly left-resolving}, consider the set \[S=\{(\alpha,A,\beta)\ |\ \alpha,\beta\in\awstar\ \mbox{and}\ A\in\acfrg{\alpha}\cap\acfrg{\beta}\ \mbox{with}\ A\neq\emptyset\}\cup\{0\}.\]
A binary operation on $S$ is defined as follows: $s\cdot 0= 0\cdot s=0$ for all $s\in S$ and, given $s=(\alpha,A,\beta)$ and $t=(\gamma,B,\delta)$ in $S$, \[s\cdot t=\left\{\begin{array}{ll}
(\alpha\gamma ',r(A,\gamma ')\cap B,\delta), & \mbox{if}\ \  \gamma=\beta\gamma '\ \mbox{and}\ r(A,\gamma ')\cap B\neq\emptyset,\\
(\alpha,A\cap r(B,\beta '),\delta\beta '), & \mbox{if}\ \  \beta=\gamma\beta '\ \mbox{and}\ A\cap r(B,\beta ')\neq\emptyset,\\
0, & \mbox{otherwise}.
\end{array}\right. \]
If for a given $s=(\alpha,A,\beta)\in S$ we define $s^*=(\beta,A,\alpha)$, then the set $S$, endowed with the operation above, is an inverse semigroup with zero element $0$ (\cite{MR3648984}, Proposition 3.4), whose semilattice of idempotents is \[E(S)=\{(\alpha, A, \alpha) \ | \ \alpha\in\awstar \ \mbox{and} \ A\in\acfra\}\cup\{0\}.\]

The natural order in the semilattice $E(S)$ is described in the next proposition.

\begin{proposition}[\cite{MR3648984}, Proposition 4.1]
	Let $p=(\alpha, A, \alpha)$ and $q=(\beta, B, \beta)$ be non-zero elements in $E(S)$. Then $p\leq q$ if and only if $\alpha=\beta\alpha'$ and $A\subseteq r(B,\alpha')$.
\end{proposition}

\subsection{Filters in $E(S)$}\label{subsection:filters.E(S)}

For $E$ a (meet) semilattice with $0$, there is a bijection between the set of filters in $E$ (upper sets that are closed under meets and that do not contain $0$) and the set $\hat{E}_0$ of characters of $E$ (zero and meet-preserving nonzero maps from $E$ to the Boolean algebra $\{0,1\}$). Under the topology of pointwise convergence in $\hat{E}_0$, the closure of the subset $\hat{E}_\infty$ of characters that correspond to ultrafilters in $E$ is denoted by $\hat{E}_{tight}$, and is called the \emph{tight spectrum} of $E$; its elements are the \emph{tight characters of $E$}, and their corresponding filters are \emph{tight filters}. See \cite[Section 12]{MR2419901} for details.

For a given labelled space $\lspace$ that is weakly left-resolving and normal, we now review a description of the tight spectrum of $E(S)$ given in \cite{MR3648984}: let us begin by dropping the normality and considering a labelled space that is just weakly left-resolving.

Let $\alpha\in\awleinf$ and $\{\ftg{F}_n\}_{0\leq n\leq|\alpha|}$ (understanding that ${0\leq n\leq|\alpha|}$ means $0\leq n<\infty$ when $\alpha\in\awinf$) be a family such that $\ftg{F}_n$ is a filter in $\acfrg{\alpha_{1,n}}$ for every $n>0$ and $\ftg{F}_0$ is either a filter in $\acf$ or $\ftg{F}_0=\emptyset$. The family $\{\ftg{F}_n\}_{0\leq n\leq|\alpha|}$ is said to be \emph{complete for} $\alpha$ if
\[\ftg{F}_n = \{A\in \acfrg{\alpha_{1,n}} \ | \ r(A,\alpha_{n+1})\in\ftg{F}_{n+1}\}\]
for all $n\geq0$.

It is worth pointing out that in the case of a labelled space that is weakly left-resolving and normal, if a filter in a complete family is an ultrafilter, then all filters in the family coming before it are also ultrafilters \cite[Proposition 5.7]{MR3648984}.

\begin{theorem}[\cite{MR3648984}, Theorem 4.13]\label{thm.filters.in.E(S)}
	Let $\lspace$ be a weakly left-resolving labelled space and $S$ be its associated inverse semigroup. Then there is a bijective correspondence between filters in $E(S)$ and pairs $(\alpha, \{\ftg{F}_n\}_{0\leq n\leq|\alpha|})$, where $\alpha\in\awleinf$ and $\{\ftg{F}_n\}_{0\leq n\leq|\alpha|}$ is a complete family for $\alpha$.
\end{theorem}

Filters are of \emph{finite type} if they are associated with pairs $(\alpha, \{\ftg{F}_n\}_{0\leq n\leq|\alpha|})$ for which $|\alpha|<\infty$, and of \emph{infinite type} otherwise.

A filter $\xi$ in $E(S)$ with associated labelled path $\alpha\in\awleinf$ is sometimes denoted by $\xia$ to stress the word $\alpha$; in addition, the filters in the complete family associated with $\xia$ will be denoted by $\xi^{\alpha}_n$ (or simply $\xi_n$). Specifically,
\begin{align}\label{eq.defines.xi_n}
	\xi^{\alpha}_n=\{A\in\acf \ | \ (\alpha_{1,n},A,\alpha_{1,n}) \in \xia\}.
\end{align}

\begin{remark}\label{remark.when.in.xialpha}
	It follows from \cite[Propositions 4.4 and 4.8]{MR3648984} that for a filter $\xi^\alpha$ in $E(S)$ and an element $(\beta,A,\beta)\in E(S)$ we have $(\beta,A,\beta)\in\xi^{\alpha}$ if and only if $\beta$ is a beginning of $\alpha$ and $A\in\xi_{|\beta|}$.
\end{remark}

\begin{theorem}[\cite{MR3648984}, Theorems 5.10 and 6.7]
	\label{thm.tight.filters.in.es}
	Let $\lspace$ be a weakly left-resolving, normal labelled space and $S$ be its associated inverse semigroup. Then the tight filters in $E(S)$ are:
	\begin{enumerate}[(i)]
		\item The filters of infinite type for which the non-empty elements of their associated complete families are ultrafilters. 
		\item The filters of finite type $\xia$ such that $\xi_{|\alpha|}$ is an ultrafilter in $\acfra$ and for each  $A\in\xi_{|\alpha|}$ at least one of the following conditions hold:
		\begin{enumerate}[(a)]
			\item $\lbf(A\dgraph^1)$ is infinite.
			\item There exists $B\in\acfra$ such that $\emptyset\neq B\subseteq A\cap \dgraph^0_{sink}$.
		\end{enumerate}
	\end{enumerate}
\end{theorem}

The set $\ftight$ of tight filters in $E(S)$ is endowed with the topology induced from the topology of pointwise convergence of characters of $E(S)$, and is thus treated as being (homeomorphic to) the tight spectrum of $E(S)$. For $\alpha\in\awstar$, we denote by $\ftightw{\alpha}$ the set of all tight filters in $E(S)$ for which the associated word is $\alpha$. Also, in what follows $\filt$ denotes the set of all filters in $E(S)$.

\subsection{Filter surgery in E(S)}\label{subsection.filter.surgery}

From now on, $\lspace$ stands for a weakly left-resolving \emph{normal} labelled space unless stated otherwise, and all sets $X_\alpha$, as presented at the end of Subsection \ref{subsect.labelled.spaces}, are considered as topological spaces, with the topology given by convergence of filters (that is, the pointwise convergence of the corresponding characters). For all omitted details and proofs in this subsection, see \cite[Section 4]{MR3680957}.

Given $\alpha,\beta\in\awplus$ such that $\alpha\beta\in\awplus$, the relative range map $\newf{r(\,\cdot\,,\beta)}{\acfrg{\alpha}}{\acfrg{\alpha\beta}}$ is a morphism of Boolean algebras and, therefore, we have its dual morphism \[\newf{f_{\alpha[\beta]}}{X_{\alpha\beta}}{X_{\alpha}}\] given by $f_{\alpha[\beta]}(\ft)=\{A\in\acfra \ | \ r(A,\beta)\in\ft\}$ (think of this as the map on ultrafilters induced by cutting $\beta$ from the end of the labelled path $\alpha\beta$).

When $\alpha=\eword$,  if $\ft\in\acfrg{\beta}$ then $\{A\in\acf \ | \ r(A,\beta)\in\ft\}$ is either an ultrafilter in $\acf=\acfrg{\eword}$ or is the empty set, and we can therefore consider $\newf{f_{\eword[\beta]}}{X_{\beta}}{X_{\eword}\cup\{\emptyset\}}$. These functions $f_{\alpha[\beta]}$ are continuous, and $f_{\alpha[\beta\gamma]}=f_{\alpha[\beta]}\circ f_{\alpha\beta[\gamma]}$.

We now review functions described in \cite{MR3680957} that generalize two operations that can easily be done with paths on a graph $\dgraph$: gluing paths, that is, given $\mu$ and $\nu$ paths on $\dgraph$ such that $r(\mu)=s(\nu)$, it is easy to see that $\mu\nu$ is a new path on $\dgraph$; and cutting paths, that is, given a path $\mu\nu$ on $\dgraph$ then $\nu$ is also a path on the graph.

In the context of labelled spaces, we have an extra layer of complexity because filters in $E(S)$ are described not only by a labelled path but also by a complete family of filters associated with it, by Theorem \ref{thm.filters.in.E(S)}. When we cut or glue labelled paths, the Boolean algebras where the filters lie change because they depend on the labelled path. We also note that, since we are only interested in tight filters in $E(S)$, it is enough to consider families consisting only of ultrafilters, by Theorem \ref{thm.tight.filters.in.es}.

Let us begin with the gluing: for composable labelled paths $\alpha\in\awplus$ and $\beta\in\awstar$ (that is, such that $\alpha\beta\in\awplus$), consider the subspace $X_{(\alpha)\beta}$ of $X_\beta$ given by \[X_{(\alpha)\beta}=\{\ft\in X_\beta \ |\ r(\alpha\beta)\in\ft\}.\] There is then a continuous map \[\newf{g_{(\alpha)\beta}}{X_{(\alpha)\beta}}{X_{\alpha\beta}}\] on ultrafilters induced by gluing $\alpha$ at the beginning of the labelled path $\beta$ given by
\begin{align}\label{eq:def.g}
	g_{(\alpha)\beta}(\ft)=\{C\cap r(\alpha\beta)\ |\ C\in\ft\}.
\end{align}

The following simple result will be needed later on.

\begin{lemma}\label{lemma:A.belongs.g(F)}
	Suppose that $A\in\acfrg{\gamma}$ and $\ft\in X_{(a)\gamma}$. Then, $A\in\ft$ if and only if $A\cap r(a\gamma)\in g_{(a)\gamma}(\ft)$.
\end{lemma}

\begin{proof}
	If $A\in\ft$ then $A\cap r(a\gamma)\in g_{(a)\gamma}(\ft)$ from the definition of $g_{(a)\gamma}(\ft)$. For the converse, suppose that $A\cap r(a\gamma)\in g_{(a)\gamma}(\ft)$. Since $g_{(a)\gamma}(\ft)\subseteq \ft$, $A\in\acfrg{\gamma}$ and $\ft$ is a filter in $\acfrg{\gamma}$, it follows that $A\in \ft$.
\end{proof}

Now for composable labelled paths $\alpha\in\awplus$ and $\beta\in\awleinf$, let $\ftightw{(\alpha)\beta}$ be the subspace of $\ftightw{\beta}$ given by \[ \ftightw{(\alpha)\beta}=\{\xi\in\ftightw{\beta}\ |\ \xi_0\in X_{(\alpha)\eword}\}.\]
We can then define a gluing map \[\newf{G_{(\alpha)\beta}}{\ftightw{(\alpha)\beta}}{\ftightw{\alpha\beta}}\] taking a tight filter $\xi\in\ftightw{(\alpha)\beta}$ to the tight filter $\eta\in\ftightw{\alpha\beta}$, whose complete family of (ultra)filters is obtained by gluing and cutting labelled paths appropriately, as follows:
\begin{itemize}
	\item If $\beta=\eword$, \[ \eta_{|\alpha|}=g_{(\alpha)\eword}(\xi_0)=\{C\cap r(\alpha)\ |\ C\in\xi_0\} \] and, for for $0\leq i< |\alpha|$, \[ \eta_i=f_{\alpha_{1,i}[\alpha_{i+1,|\alpha|}]}(\eta_{|\alpha|})=\{D\in\acfrg{\alpha_{1,i}}\ |\ r(D,\alpha_{i+1,|\alpha|})\in\eta_{|\alpha|}\};\]
	
	\item If $\beta\neq\eword$, for $1\leq n\leq|\beta|$ (or $n<|\beta|$ if $\beta$ is infinite)
	\[ \eta_{|\alpha|+n} = g_{(\alpha)\beta_{1,n}}(\xi_n) = \{C\cap r(\alpha\beta_{1,n})\ |\ C\in\xi_n\}\] and, for $0\leq i\leq|\alpha|$,
	\[\eta_i = f_{\alpha_{1,i}[\alpha_{i+1,|\alpha|}\beta_1]}(\eta_{|\alpha|+1})
	=\{D\in\acfrg{\alpha_{1,i}}\ |\ r(D,\alpha_{i+1,|\alpha|}\beta_1)\in\eta_{|\alpha|+1}\}.\]
\end{itemize}
Finally, for $\alpha=\eword$ set $\ftightw{(\eword)\beta}=\ftightw{\beta}$ and let $G_{(\eword)\beta}$ be the identity function on $\ftightw{\beta}$.

Next, we describe the cutting: for composable labelled paths $\alpha\in\awplus$ and $\beta\in\awstar$, there is a continuous map \[\newf{h_{[\alpha]\beta}}{X_{\alpha\beta}}{X_{(\alpha)\beta}}\] induced on ultrafilters by cutting $\alpha$ from the beginning of $\alpha\beta$ given by
\begin{align}\label{eq:def.h}
	h_{[\alpha]\beta}(\ft)=\usetr{\ft}{\acfrg{\beta}}=\{C\in\acfrg{\beta}\ |\ D\leq C\text{ for some }D\in\ft\}.
\end{align}

\begin{lemma}\label{lemma:A.belongs.h(F)}
	Suppose that $A\in \acfrg{a\gamma}$ and $\ft\in X_{a\gamma}$. Then, $A\in\ft$ if and only if $A\in h_{[a]\gamma}(\ft)$.
\end{lemma}

\begin{proof}
	If $A\in\ft$ then $A\in h_{[a]\gamma}(\ft)$ from the definition of $h_{[a]\gamma}(\ft)$. Now, if $A\in h_{[a]\gamma}(\ft)$, using that $\ft=g_{(a)\gamma}(h_{[a]\gamma}(\ft))=\{B\cap r(a\gamma)\,|\, B\in h_{[a]\gamma}(\ft)\}$, since $A\subseteq r(a\gamma)$, we conclude that $A\in\ft$.
\end{proof}

For composable labelled paths $\alpha\in\awplus$ and $\beta\in\awleinf$, these give rise to a cutting map \[\newf{H_{[\alpha]\beta}}{\ftightw{\alpha\beta}}{\ftightw{(\alpha)\beta}}\] that takes a tight filter $\xi\in\ftightw{\alpha\beta}$ to the tight filter $\eta\in\ftightw{(\alpha)\beta}$ such that, for all $n$ with $0\leq n\leq |\beta|$, \[ \eta_n=h_{[\alpha]\beta_{1,n}}(\xi_{n+|\alpha|}).\] For $\alpha=\eword$ define $H_{[\eword]\beta}$ to be the identity function over $\ftightw{\beta}$.


\begin{theorem}[\cite{MR3680957}, Theorem 4.17]
	\label{thm.H-class.G-class.inverses}
	Suppose the labelled space $\lspace$ is weakly left-resolving and normal, and let $\alpha\in\awplus$ and $\beta\in\awleinf$ be such that $\alpha\beta\in\awleinf$. Then $H_{[\alpha]\beta}=(G_{(\alpha)\beta})^{-1}$.
\end{theorem}

\begin{theorem}[\cite{MR3680957}, Lemmas 4.13 and 4.16]
	\label{thm.compositions.G.and.H}
	Suppose the labelled space $\lspace$ is weakly left-resolving and normal, and let $\alpha,\beta\in\awplus$ and $\gamma\in\awleinf$ be such that $\alpha\beta\gamma\in\awleinf$. Then $G_{(\alpha\beta)\gamma}=G_{(\alpha)\beta\gamma}\circ G_{(\beta)\gamma}$ and $H_{[\beta]\gamma}\circ H_{[\alpha]\beta\gamma}=H_{[\alpha\beta]\gamma}$.
\end{theorem}

\subsection{The C*-algebra of a labelled space} Let $\lspace$ be a weakly left-resolving, normal labelled space. The \emph{C*-algebra associated with} $\lspace$, denoted by $C^*\lspace$, is the universal $C^*$-algebra generated by projections $\{p_A \ | \ A\in \acf\}$ and partial isometries $\{s_a \ | \ a\in\alf\}$ subject to the relations
\begin{enumerate}[(i)]
	\item $p_{A\cap B}=p_Ap_B$, $p_{A\cup B}=p_A+p_B-p_{A\cap B}$ and $p_{\emptyset}=0$, for every $A,B\in\acf$;
	\item $p_As_a=s_ap_{r(A,a)}$, for every $A\in\acf$ and $a\in\alf$;
	\item $s_a^*s_a=p_{r(a)}$ and $s_b^*s_a=0$ if $b\neq a$, for every $a,b\in\alf$;
	\item For every $A\in\acf$ for which $0<\card{\lbf(A\dgraph^1)}<\infty$ and there does not exist $B\in\acf$ such that $\emptyset\neq B\subseteq A\cap \dgraph^0_{sink}$,
	\[p_A=\sum_{a\in\lbf(A\dgraph^1)}s_ap_{r(A,a)}s_a^*.\]
\end{enumerate}

For each word $\alpha=a_1a_2\cdots a_n$, define $s_\alpha=s_{a_1}s_{a_2}\cdots s_{a_n}$; we also set $s_{\eword}=1$, where $\eword$ is the empty word. We observe that $s_\eword$ does not belong to $C^*\lspace$ unless it is unital -- we work with $s_{\eword}$ to simplify our statements. For example, $s_{\eword}p_{A}s_{\eword}^*$ means $p_{A}$. We never use $s_{\eword}$ alone.

\begin{proposition}\label{prop:closure.span}
	Let $\lspace$ be a weakly left-resolving normal labelled space. Then \[C^*\lspace = \overline{\mathrm{span}}\{s_\alpha p_A s_\beta^* \ | \ \alpha,\beta\in\awstar \ \mbox{and} \ A\in\acfrg{\alpha}\cap\acfrg{\beta}\}.\]
\end{proposition}

For details, see \cite{MR3614028} and \cite{MR3680957}.

\section{The groupoid associated with a labelled space}

In \cite[Section 4]{MR2419901} a certain action of inverse semigroups on their tight spectra is constructed, from which one can associate a groupoid of germs. In this section we give an isomorphism between this groupoid of germs and a kind of boundary path groupoid (such as in \cite{MR2184052} and \cite{MR2301938}). This description will facilitate the study of the groupoid C*-algebra in Section \ref{section:cstar}.

We begin by constructing, in the present context, the groupoid $\mathcal{G}_{tight}$ as in \cite{MR2419901}. Let $\lspace$ be a labelled space, with associated inverse semigroup $S$. For each idempotent $e\in E(S)$, define $D_e=\{\phi\in\hat{E}_{tight}\ |\ \phi(e)=1\}$ and $\Omega=\{(s,\phi)\in S\times\hat{E}_{tight}\ |\ \phi\in D_{s^*s} \}$. The action $\theta$ of $S$ on $\hat{E}_{tight}$ is given by
\begin{align}\label{eq.def.action}
	\theta_s(\phi)(e)=\phi(s^*es).
\end{align}

The following is an equivalence relation on $\Omega$: $(s,\phi)\sim(t,\psi)$ if and only if $\phi=\psi$ and there exists $e\in E(S)$ such that $\phi\in D_{e}$ and $se=te$. Let $\mathcal{G}_{tight}=\Omega/\sim$ and denote the class of $(s,\phi)$ by $[s,\phi]$. Set
\[ \mathcal{G}^{(2)}_{tight}=\left\lbrace\left([s,\phi],[t,\psi] \right)\in\mathcal{G}_{tight}\times\mathcal{G}_{tight}\ |\ \phi=\theta_t(\psi)  \right\rbrace  \]
and for $ \left([s,\phi],[t,\psi] \right)\in\mathcal{G}^{(2)}_{tight} $ define
\[ [s,\phi]\cdot[t,\psi]=[st,\psi]. \]
Also, for $ [s,\phi]\in\mathcal{G}_{tight} $ let
\[ [s,\phi]^{-1}=[s^*,\theta_s(\phi)]. \]
Then $\mathcal{G}_{tight}$ is a groupoid with operations defined as above.

\begin{remark}\label{remark.element.of.omega}
	Given a non zero element $s=(\mu,A,\nu)\in S$ and $\phi\in\hat{E}_{tight}$, let us characterize when $\phi(s^*s)=1$. Since $s^*s=(\nu,A,\nu)$, if $\xi^{\alpha}$ is the filter on $E(S)$ associated with $\phi$ then $\phi(s^*s)=1$ if and only if $(\nu,A,\nu)\in\xi^{\alpha}$; from (\ref{eq.defines.xi_n}), this happens if and only if $\nu$ is a beginning of $\alpha$ and $A\in\xi^{\alpha}_{|\nu|}$.
\end{remark}

\begin{proposition} \label{prop:equivalence.relation.on.omega}
	Let $s=(\mu,A,\nu)$, $t=(\beta,B,\gamma)$ be two non zero elements in $S$ and $\phi$ be a tight character associated with a filter $\xi^{\alpha}$ on $E(S)$. Suppose that $(s,\phi),(t,\phi)\in\Omega$ and that $\nu$ is a beginning of $\gamma$ with $\gamma=\nu\gamma'$. Then $(s,\phi)\sim(t,\phi)$ if and only if $\beta=\mu\gamma'$.
\end{proposition}

\begin{proof}
	Suppose that $(s,\phi)\sim(t,\phi)$, then there exists $e=(\delta,N,\delta)\in E(S)$ such that $\phi\in D_e$ and $se=te$. Since $\phi \in D_e$, we have $\phi(e)=1$; therefore, Remark \ref{remark.element.of.omega} ensures $\delta$ is a beginning of $\alpha$ and $N\in\xi_{|\delta|}^\alpha$. Additionally, it follows from $(t,\phi)\in\Omega$ that $\phi\in D_{t^*t}$. Remark \ref{remark.element.of.omega} then says $\gamma$ (and thus also $\nu$, since $\gamma=\nu\gamma'$) is a beginning of $\alpha$ and hence $\delta$ is comparable with $\gamma$ (and $\nu$). There are a few cases to consider.
	
	First consider that $|\delta|\leq|\nu|\leq|\gamma|$. Writing $\nu=\delta\nu'$ and $\gamma=\delta\gamma''$ for $\mu',\gamma''\in\awstar$ we have
	\begin{align*}
		se & = (\mu,A\cap r(N,\nu'),\nu) \\
		te& = (\beta,B\cap r(N,\gamma''),\gamma).
	\end{align*}
	It follows from $se=te$ that $\gamma=\nu$, that is, $\gamma'=\eword$, hence $\beta=\mu=\mu\gamma'$ as desired.
	
	Now suppose that $|\nu|\leq|\gamma|\leq|\delta|$. In this case $\delta=\gamma\delta'=\nu\gamma'\delta'$ for $\delta'\in\awstar$, and
	\begin{align*}
		se & = (\mu\gamma'\delta',r(A,\gamma'\delta')\cap N,\delta) \\
		te & = (\beta\delta',r(B,\delta')\cap N,\delta),
	\end{align*}
	whence $\beta\delta'=\mu\gamma'\delta'$ and therefore $\beta=\mu\gamma'$.
	
	Finally, if $|\nu|<|\delta|<|\gamma|$, from the previous cases the third coordinate of $se$ is $\delta$ and the third coordinate of $te$ is $\gamma$; however, since $|\delta|<|\gamma|$, we have $\delta\neq\gamma$ and $se\neq te$, which is a contradiction.
	
	For the converse, suppose that $\beta=\mu\gamma'$ and define $e=(\gamma,r(A,\gamma')\cap B,\gamma)$. Let us first check that $e\in E(S)\setminus\{0\}$. From $(s,\phi),(t,\phi)\in \Omega$, Remark \ref{remark.element.of.omega} gives $A\in\xi^{\alpha}_{|\nu|}$ and $B\in\xi^{\alpha}_{|\gamma|}$. Using that  $|\gamma|=|\nu|+|\gamma'|$, $A\in\xi^{\alpha}_{|\nu|}$ and the completeness of the family $\{\xi^\alpha_n\}_n$ we obtain $r(A,\gamma')\in\xi^{\alpha}_{|\gamma|}$. In particular $B\cap r(A,\gamma')\in\xi^{\alpha}_{|\gamma|}$ since $\xi^{\alpha}_{|\gamma|}$ is a filter and so $\emptyset\neq r(A,\gamma')\cap B\in\acfrg{\gamma}$, hence $e\in E(S)\setminus\{0\}$. Now
	\begin{align*}
		se & =(\mu\gamma',r(A,\gamma')\cap B,\gamma) \\
		te & = (\beta,r(A,\gamma')\cap B,\gamma),
	\end{align*}
	so that $se=te$ and $(s,\phi)\sim(t,\phi)$.
	
\end{proof}

\begin{remark}\label{remark.equivalence.relation.on.omega}
	Let $s=(\mu,A,\nu)$, $t=(\beta,B,\gamma)$ be such that $(s,\phi),(t,\phi)\in\Omega$ for some $\phi\in\hat{E}_{tight}$. If $\xia$ is the filter associated with $\phi$ then $\nu$ and $\gamma$ are both beginnings of $\alpha$ so that they are comparable. It follows from Proposition \ref{prop:equivalence.relation.on.omega} that $(s,\phi)\sim(t,\phi)$ if and only if $\mu$ and $\nu$ are beginnings of $\beta$ and $\gamma$ respectively with the same ending, or the other way around.
\end{remark}

We now focus our attention in defining an analogue to the boundary path groupoid of a graph (\cite{MR2184052}, \cite{MR2301938}) in the setting of labelled spaces. Let $s=(\mu,A,\nu)\in S$ and $\phi\in\hat{E}_{tight}$ be such that $(s,\phi)\in\Omega$. If $\xia$ is the filter corresponding to $\phi$ then by Remark \ref{remark.element.of.omega}, $\nu $ is a beginning of $\alpha$, that is $\alpha=\nu\alpha'$ for some $\alpha'\in\awleinf$.

As with the case of graphs, a candidate to be associated with the class $[s,\phi]$ would be the triple $(\mu\alpha',|\mu|-|\nu|,\nu\alpha')$, but doing so would ignore the information contained in the family $ \{\xi^{\alpha}_n\}_n $ of ultrafilters. In order not to lose this information we cut and glue tight filters adequately, using filter surgery as described in Subsection \ref{subsection.filter.surgery}, to build a new tight filter $\eta^{\mu\alpha'}$ from $\xia$ by cutting off $\nu$ from the beginning of $\alpha=\nu\alpha'$ and gluing $\mu$ in its place.



\begin{proposition}
	Let $\lspace$ be a weakly left-resolving, normal labelled space and define
	\[ \Gamma=\{(\xi^{\alpha\gamma},|\alpha|-|\beta|,\eta^{\beta\gamma})\in \ftight\times\mathbb{Z}\times\ftight\ |\ H_{[\alpha]\gamma}(\xi^{\alpha\gamma})=H_{[\beta]\gamma}(\eta^{\beta\gamma})\}. \]
	Then $\Gamma$ is a groupoid with product given by
	\[ (\xi,m,\eta)(\eta,n,\rho)=(\xi,m+n,\rho) \]
	and an inverse given by
	\[ (\xi,m,\eta)^{-1}=(\eta,-m,\xi). \]
\end{proposition}

\begin{proof}	
	The only difficulty lies in proving that the product $(\xi,m+n,\rho)$ is an element of $\Gamma$. On the one hand, from $(\xi,n,\eta)\in\Gamma$ we can write $\xi=\xi^{\alpha\gamma}$ and $\eta=\eta^{\beta\gamma}$ with $|\alpha|-|\beta|=m$; on the other hand, since $(\eta,n,\rho)\in\Gamma$ we have $\eta=\eta^{\beta'\gamma'}$ and $\rho=\rho^{\delta\gamma'}$ with $|\beta'|-|\delta|=n$. Now $\eta=\eta^{\beta\gamma}=\eta^{\beta'\gamma'}$, so $\beta\gamma=\beta'\gamma'$; this ensures that either $\beta$ is a beginning of $\beta'$ or $\beta'$ is a beginning of $\beta$.
	
	In the case that $\beta$ is a beginning of $\beta'$, say $\beta'=\beta\gamma''$, \[\beta\gamma=\beta'\gamma'=\beta\gamma''\gamma',\] so that $\gamma=\gamma''\gamma'$. Therefore $\xi=\xi^{\alpha\gamma}=\xi^{(\alpha\gamma'')\gamma'}$, $\rho=\rho^{\delta\gamma'}$, and
	\begin{align*}
		m+n & = (|\alpha|-|\beta|)+(|\beta'|-|\delta|) \\
		& = |\alpha|-|\beta|+|\beta\gamma''|-|\delta| \\
		& = |\alpha\gamma''|-|\delta|.
	\end{align*}
	Additionally, using that $H_{[\alpha]\gamma}(\xi)=H_{[\beta]\gamma}(\eta)$ and $H_{[\beta']\gamma'}(\eta)=H_{[\delta]\gamma'}(\rho)$,
	\begin{align*}
		H_{[\alpha\gamma'']\gamma'}(\xi) & = (H_{[\gamma'']\gamma'}\circ H_{[\alpha]\gamma''\gamma'})(\xi) = (H_{[\gamma'']\gamma'}\circ H_{[\alpha]\gamma})(\xi) \\
		& = (H_{[\gamma'']\gamma'}\circ H_{[\beta]\gamma})(\eta) = (H_{[\gamma'']\gamma'}\circ H_{[\beta]\gamma''\gamma'})(\eta) \\
		& = H_{[\beta\gamma'']\gamma'}(\eta) = H_{[\beta']\gamma'}(\eta)\\
		& = H_{[\delta]\gamma'}(\rho)
	\end{align*}
	whence $(\xi,m+n,\rho)\in \Gamma$. The case that $\beta'$ is a beginning of $\beta$ is analogous.	
\end{proof}

\begin{lemma}\label{lemma:intersection.in.filter}
	Let $\lspace$ be a weakly left-resolving, normal labelled space. Suppose that $\alpha,\beta\in\awstar$ are such that $\alpha\beta\in\awstar$, and that $\ft$ is a filter on $\acfrg{\alpha\beta}$. If $A\in\ft$ and $C\in \usetr{\ft}{\acfrg{\beta}}$, then $A\cap C\in\ft$.
\end{lemma}

\begin{proof}
	Notice that $A\cap C\in\acfrg{\alpha\beta}$ since $\acf$ is closed under intersections and $A\cap C\scj A\scj r(\alpha\beta)$. Also, there exists $X\in \ft$ such that $X\scj C$. Since $A\cap X\scj A\cap C$ and $A\cap X\in \ft$, we have that $A\cap C\in \ft$.
\end{proof}

\begin{lemma}\label{lemma:cut.and.glue}
	Let $\lspace$ be a weakly left-resolving, normal labelled space. Let $(t,\phi)$ be an element of $\Omega$ with $t=(\beta,A,\gamma)$ and $\phi\in\hat{E}_{tight}$ be a character associated with $\xia$ so that $\alpha=\gamma\alpha'$ for some $\alpha'\in\awleinf$. Then $\eta=(G_{(\beta)\alpha'}\circ H_{[\gamma]\alpha'})(\xia)$, where $\alpha=\gamma\alpha'$, is well defined, and $\eta$ is the character associated with $\theta_t(\phi)$ given by (\ref{eq.def.action}).
\end{lemma}

\begin{proof}
	Let $(\delta,D,\delta)\in E(S)$. Then, using Remark \ref{remark.element.of.omega},
	\[ \theta_t(\phi)(\delta,D,\delta) = \phi((\gamma,A,\beta)(\delta,D,\delta)(\beta,A,\gamma)) \]
	
	\[ 	=\left\{	\begin{array}{ll} \phi(\gamma\delta',r(A,\delta')\cap D,\gamma\delta') & \text{if }\delta = \beta\delta' \\
	\phi(\gamma,A\cap r(D,\beta'),\gamma) & \text{if }\beta=\delta\beta' \\
	0 & \text{otherwise}
	\end{array} \right. \]
	
	\[ 	=\left\{	\begin{array}{ll} \left[\delta'\text{ is a beginning of }\alpha'\text{ and }r(A,\delta')\cap D\in\xi^{\alpha}_{|\gamma\delta'|}\right] & \text{if }\delta = \beta\delta' \\
	\lbrack A\cap r(D,\beta')\in\xi^{\alpha}_{|\gamma|} \rbrack & \text{if }\beta=\delta\beta' \\
	0 & \text{otherwise,}
	\end{array} \right. \]	
	where $[\ ]$ represents the boolean function that returns $0$ if the argument is false and $1$ if it is true. Note that for $\theta_t(\phi)(\delta,D,\delta)=1$ to hold, it is necessary for $\delta$ to be a beginning of $\beta\alpha'$. There is a also a condition on $D$ that depends on whether $\delta = \beta\delta'$ or $\beta=\delta\beta'$.
	
	To see that $\eta=(G_{(\beta)\alpha'}\circ H_{[\gamma]\alpha'})(\xia)$ is well defined first notice that $(\beta,A,\gamma)$ is an element of $S\lspace$ so that $\emptyset\neq A\subseteq r(\beta)\cap r(\gamma)$. If $\beta=\eword$, then the domain of $G_{(\beta)\alpha'}$ is $\ftightw{\alpha}$, which contains the range of $H_{[\gamma]\alpha'}$. If $\beta\neq\eword$ then by Remark \ref{remark.element.of.omega} $A\in \xi^{\alpha}_{|\gamma|}$, and since $A\subseteq r(\beta)$, this implies that $r(\beta)\in H_{[\gamma]\alpha'}(\xia)_0$ by the definition of $H$. Finally, from the definition of $\ftightw{(\beta)\alpha'}$ we conclude that $H_{[\gamma]\alpha'}(\xia)\in \ftightw{(\beta)\alpha'}$, which is the domain of $G_{(\beta)\alpha'}$.
	
	Let $\psi$ be the character associated with $\eta=(G_{(\beta)\alpha'}\circ H_{[\gamma]\alpha'})(\xia)$, then
	\[ \psi(\delta,D,\delta)=\left[\delta\text{ is a beginning of }\beta\alpha'\text{ and }D\in\eta_{|\delta|} \right].\]
	
	If $\delta = \beta\delta'$ then, by (\ref{eq:def.g}) and (\ref{eq:def.h}),
	\dmal{\eta_{|\delta|} & =(g_{(\beta)\delta'}\circ h_{[\gamma]\delta'})(\xi_{|\gamma\delta'|})\\
		& = g_{(\beta)\delta'}\left(\usetr{\xi_{|\gamma\delta'|}}{\acfrg{\delta'}}\right)\\
		& = \left\{C\cap r(\beta\delta')\ |\ C\in\usetr{\xi_{|\gamma\delta'|}}{\acfrg{\delta'}}\right\}.}
	
	We claim that $D\in\eta_{|\delta|}$ if and only if $D\cap r(A,\delta')\in\xi^{\alpha}_{|\gamma\delta'|}$. On the one hand, if $D\in\eta_{|\delta|}$, then there exists $C\in\usetr{\xi_{|\gamma\delta'|}}{\acfrg{\delta'}}$ such that $D=C\cap r(\beta\delta')$. Since $(\beta,A,\gamma)\in S$ we have $A\in\acfrg{\beta}\cap \acfrg{\gamma}$ and so $r(A,\delta')\subseteq r(\beta\delta')$. Now, $D\cap r(A,\delta')=C\cap r(\beta\delta')\cap r(A,\delta')=C\cap r(A,\delta')\subseteq r(\gamma\delta')$, which is an element of $\xi^{\alpha}_{|\gamma\delta'|}$ by Lemma \ref{lemma:intersection.in.filter}.
	
	On the other hand if $D\cap r(A,\delta')\in\xi^{\alpha}_{|\gamma\delta'|}$, since $D\in\acfrg{\delta}\subseteq\acfrg{\delta'}$ and $r(A,\delta')\cap D\subseteq D$, then $D\in\usetr{\xi_{|\gamma\delta'|}}{\acfrg{\delta'}}$ and $D=D\cap r(\beta\delta')\in\eta_{|\delta|}$.
	
	Now suppose that $\beta=\delta\beta'$. In this case by the definitions of $G$ and $H$ given in Subsection \ref{subsection.filter.surgery},
	\dmal{\eta_{|\delta|} & = (f_{\delta[\beta']}\circ g_{(\beta)\eword}\circ h_{[\gamma]\eword})(\xi^{\alpha}_{|\gamma|})\\
		&=(f_{\delta[\beta']}\circ g_{(\beta)\eword})\left(\usetr{\xi^{\alpha}_{|\gamma|}}{\acf}\right)\\
		&=f_{\delta[\beta']}\left(\left\{C\cap r(\beta)\ |\ C\in\usetr{\xi^{\alpha}_{|\gamma|}}{\acf}\right\}\right)\\
		&=\left\{F\in\acfrg{\delta}\ |\ r(F,\beta')\in \left\{C\cap r(\beta)\ |\ C\in\usetr{\xi^{\alpha}_{|\gamma|}}{\acf}\right\} \right\}.
	}
	
	We claim that $D\in\eta_{|\delta|}$ if and only if $A\cap r(D,\beta')\in\xi^{\alpha}_{|\gamma|}$. In fact, if $D\in\eta_{|\delta|}$ then there exists $C\in\usetr{\xi^{\alpha}_{|\gamma|}}{\acf} $ such that $r(D,\beta')=C\cap r(\beta)$. Since $A\subseteq r(\beta)$ and $A\in\xi^{\alpha}_{|\gamma|}$, we have that $A\cap r(D,\beta')=A\cap C\cap r(\beta)=A\cap C$ which is an element of $\xi^{\alpha}_{|\gamma|}$ by Lemma \ref{lemma:intersection.in.filter}.
	
	On the other hand, if $A\cap r(D,\beta')\in\xi^{\alpha}_{|\gamma|}$, since $A\in\xi^{\alpha}_{|\gamma|}$, by choosing $C=r(D,\beta')$ in the expression for $\eta_{|\delta|}$, we see that $D\in\eta_{|\delta|}$.
	
	By comparing the formulas for $\psi$ and $\theta_t(\phi)$ we conclude that they are equal, that is, $\eta$ is the filter associated with $\theta_t(\phi)$.
	
\end{proof}

\begin{theorem}\label{thm:isomorphism.groupoid}
	Let $\lspace$ be a weakly left-resolving, normal labelled space. Then
	\[\newfd{\Phi}{\mathcal{G}_{tight}}{\Gamma}{[t,\phi]}{((G_{(\beta)\alpha'}\circ H_{[\gamma]\alpha'})(\xia),|\beta|-|\gamma|,\xia)}\]
	is a well defined isomorphism of groupoids where $\xia$ is the filter associated with $\phi$ and $t=(\beta,A,\gamma)$ is such that $(\gamma,A,\gamma)\in\xia$, so that $\alpha=\gamma\alpha'$ for some $\alpha'\in\awleinf$.
\end{theorem}

\begin{proof}
	We divide the proof in several steps.
	
	\bigskip 
	
	\noindent $\bullet$ $\Phi$ is well defined:
	
	\bigskip 
	
	For $t$ and $\phi$ as in the statement, $(G_{(\beta)\alpha'}\circ H_{[\gamma]\alpha'})(\xia)$ is well defined by Lemma \ref{lemma:cut.and.glue}. Also, by Theorem \ref{thm.H-class.G-class.inverses},
	\[ (H_{[\beta]\alpha'}\circ G_{(\beta)\alpha'}\circ H_{[\gamma]\alpha'})(\xia) = H_{[\gamma]\alpha'}(\xia) \]
	so that $((G_{(\beta)\alpha'}\circ H_{[\gamma]\alpha'})(\xia),|\beta|-|\gamma|,\xia)\in\Gamma$.
	
	Now, let $s=(\mu,B,\nu)\in S\lspace$ be such that $\phi\in D_{s^*s}$ and $[s,\phi]=[t,\phi]$. By Remark \ref{remark.equivalence.relation.on.omega}, it is sufficient to suppose that $\gamma=\nu\gamma'$ and $\beta=\mu\gamma'$  for some $\gamma'\in\awstar$. Then
	\[ |\beta|-|\gamma| = |\mu\gamma'|-|\nu\gamma'|=|\mu|-|\nu| \]
	and, by Theorems \ref{thm.H-class.G-class.inverses} and \ref{thm.compositions.G.and.H},
	\dmal{(G_{(\beta)\alpha'}\circ H_{[\gamma]\alpha'})(\xia) & = (G_{(\mu)\gamma'\alpha'}\circ G_{(\gamma')\alpha'}\circ H_{[\gamma']\alpha'}\circ H_{[\nu]\gamma'\alpha'})(\xia) \\
		& = (G_{(\mu)\gamma'\alpha'}\circ H_{[\nu]\gamma'\alpha'})(\xia) }
	so that $\Phi$ does not depend on the representative.
	
	\bigskip 
	
	\noindent $\bullet$ $\Phi$ is injective:
	
	\bigskip 
	
	Suppose that $\Phi[t,\phi]=\Phi[s,\psi]$, where $t=(\beta,A,\gamma)$ and $s=(\mu,B,\nu)$. Since the third entries of $\Phi[t,\phi]$ and $\Phi[s,\psi]$ are the filters associated with $\phi$ and $\psi$ respectively, for the equality  $\Phi[t,\phi]=\Phi[s,\psi]$ to be true it is necessary that $\phi=\psi$. Let $\xia$ be the filter associated with $\phi$. By Remark \ref{remark.element.of.omega}, $\gamma$ and $\nu$ are beginnings of $\alpha$ so that they are comparable. Without loss of generality, suppose that $\gamma=\nu\gamma'$ for some $\gamma'\in\awstar$.
	
	By the definitions of $G$ and $H$, the first entries of $\Phi[t,\phi]$ and $\Phi[s,\psi]$ are filters with associated words $\beta\alpha'$ and $\mu\gamma'\alpha'$ respectively. From $\Phi[t,\phi]=\Phi[s,\psi]$, it follows that $\beta\alpha'=\mu\gamma'\alpha'$ and hence $\beta=\mu\gamma'$. By Proposition \ref{prop:equivalence.relation.on.omega}, $[t,\phi]=[s,\phi]$ and $\Phi$ is injective.
	
	\bigskip 
	
	\noindent $\bullet$ $\Phi$ is surjective:
	
	\bigskip 
	
	If $(\eta^{\tilde{\alpha}},m,\xia)\in\Gamma$ then there exists $\beta,\gamma\in\awstar$ and $\alpha'\in\awleinf$ such that $\tilde{\alpha}=\beta\alpha'$, $\alpha=\gamma\alpha'$,  $m=|\beta|-|\gamma|$ and $H_{[\gamma]\alpha'}(\xia)=H_{[\beta]\alpha'}(\eta^{\tilde{\alpha}})$. Let $\phi$ be the filter associated with $\xia$. Suppose that $\gamma\neq\eword$ and define $t=(\beta,r(\beta)\cap r(\gamma),\gamma)$. Since $H_{[\gamma]\alpha'}(\xia)=H_{[\beta]\alpha'}(\eta^{\tilde{\alpha}})$, by \cite[Proposition 4.6]{MR3680957} we have that $r(\beta)\cap r(\gamma)\neq\emptyset$. It follows from Theorem \ref{thm.H-class.G-class.inverses} that
	\dmal{ \Phi[t,\phi] & =((G_{(\beta)\alpha'}\circ H_{[\gamma]\alpha'})(\xia),|\beta|-|\gamma|,\xia) \\
		& = ((G_{(\beta)\alpha'}\circ H_{[\beta]\alpha'})(\eta^{\tilde{\alpha}}),|\beta|-|\gamma|,\xia) \\
		& = (\eta^{\tilde{\alpha}},m,\xia). }

	If $\gamma=\eword$ and $\alpha'\neq\eword$, we can use Theorem \ref{thm.compositions.G.and.H} to define $\gamma'=\alpha'_1$, $\beta'=\beta\alpha'_1$ and $\alpha''=\alpha'_{2,|\alpha'|}$, and repeat the above argument. If $\gamma=\alpha'=\eword$, then there exists $A\in\xi_0^{\alpha}$ and we then take $t=(\beta,r(\beta)\cap A,\gamma)$ and again show that $\Phi[t,\phi]=(\eta^{\tilde{\alpha}},m,\xia)$.
	
	\bigskip

	\noindent $\bullet$ $(\Phi\times\Phi)\left(\mathcal{G}_{tight}^{(2)}\right)\subseteq\Gamma^{(2)}$ and $\Phi$ preserves multiplication:
	
	\bigskip 
	
	Let $\left([s,\theta_t(\phi)],[t,\phi]\right)\in\mathcal{G}_{tight}^{(2)}$ with $s=(\mu,B,\nu)$, $t=(\beta,A,\gamma)$ and $\xia$ be the filter associated with $\phi$. Also, let $\eta^{\beta\alpha'}$ be the filter associated with $\theta_t(\phi)$ as in Lemma \ref{lemma:cut.and.glue}. It follows from the definition of $\Omega$ that $st\neq 0$, and so from the definition of the product given in Subsection \ref{subsection:inverse.semigroup} we have two cases to consider.
	
	\emph{Case 1:} $\beta=\nu\beta'$ for some $\beta'\in\awstar$. In this case $\beta\alpha'=\nu\beta'\alpha'$ and
	\[ st = (\mu\beta',r(B,\beta')\cap A, \gamma). \]
	
	On the one hand
	\[ \Phi[st,\phi] = ((G_{(\mu\beta')\alpha'}\circ H_{[\gamma]\alpha'})(\xia),|\mu\beta'|-|\gamma|,\xia). \]
	
	On the other hand
	\[ \Phi[s,\theta_t(\phi)] = ((G_{(\mu)\beta'\alpha'}\circ H_{[\nu]\beta'\alpha'})(\eta^{\beta\alpha'}),|\mu|-|\nu|,\eta^{\beta\alpha'}) \]
	and, by Lemma \ref{lemma:cut.and.glue},
	\[ \Phi[t,\phi] = ((G_{(\beta)\alpha'}\circ H_{[\gamma]\alpha'})(\xia),|\beta|-|\gamma|,\xia) = (\eta^{\beta\alpha'},|\beta|-|\gamma|,\xia),\]
	which implies that $(\Phi[s,\theta_t(\phi)],\Phi[t,\phi])\in \Gamma^{(2)}$.
	
	Multiplying, we obtain
	\[ \Phi[s,\theta_t(\phi)]\Phi[t,\phi] = ((G_{(\mu)\beta'\alpha'}\circ H_{[\nu]\beta'\alpha'})(\eta^{\beta\alpha'}),|\mu|-|\nu|+|\beta|-|\gamma|,\xia) .\]

	Since $\beta=\nu\beta'$,
	\[ |\mu|-|\nu|+|\beta|-|\gamma| = |\mu|-|\nu|+|\nu\beta'|-|\gamma|=|\mu\beta'|-|\gamma|. \]
	
	Also, from theorems \ref{thm.H-class.G-class.inverses} and \ref{thm.compositions.G.and.H},
	\dmal{(G_{(\mu)\beta'\alpha'}\circ H_{[\nu]\beta'\alpha'})(\eta^{\beta\alpha'})
		& = (G_{(\mu)\beta'\alpha'}\circ H_{[\nu]\beta'\alpha'}\circ G_{(\beta)\alpha'}\circ H_{[\gamma]\alpha'})(\xia) \\
		& = (G_{(\mu)\beta'\alpha'}\circ H_{[\nu]\beta'\alpha'}\circ G_{(\nu)\beta'\alpha'} \circ G_{(\beta')\alpha'}\circ H_{[\gamma]\alpha'})(\xia) \\
		& = (G_{(\mu)\beta'\alpha'} \circ G_{(\beta')\alpha'}\circ H_{[\gamma]\alpha'})(\xia) \\
		& = (G_{(\mu\beta')\alpha'} \circ H_{[\gamma]\alpha'})(\xia).
	}
	
	It follows that
	\[ \Phi[st,\phi] = \Phi[s,\theta_t(\phi)]\Phi[t,\phi]. \]
	
	\emph{Case 2:} $\nu=\beta\nu'$ for some $\nu'\in\awstar$. Since $(s,\theta_t(\phi))\in\Omega$, $\nu$ is a beginning of $\beta\alpha'$ and in this case $\alpha'=\nu'\alpha''$ for some $\alpha''\in\awleinf$. Also, we have that
	\[ st = (\mu,B\cap r(A,\nu'),\gamma\nu'). \]
	
	Now,
	\[ \Phi[st,\phi] = ( (G_{(\mu)\alpha''}\circ H_{[\gamma\nu']\alpha''})(\xia),|\mu| -|\gamma\nu'| ,\xia).\]
	
	On the other hand
	\[ \Phi[s,\theta_t(\phi)] = ((G_{(\mu)\alpha''}\circ H_{[\nu]\alpha''})(\eta^{\beta\alpha'}),|\mu|-|\nu|,\eta^{\beta\alpha'}) \]
	and
	\[ \Phi[t,\phi] = ((G_{(\beta)\alpha'}\circ H_{[\gamma]\alpha'})(\xia),|\beta|-|\gamma|,\xia) = (\eta^{\beta\alpha'},|\beta|-|\gamma|,\xia),\]	
	which again implies that $(\Phi[s,\theta_t(\phi)],\Phi[t,\phi])\in \Gamma^{(2)}$.
	
	Multiplying, we obtain
	\[ \Phi[s,\theta_t(\phi)]\Phi[t,\phi] = ((G_{(\mu)\alpha''}\circ H_{[\nu]\alpha''})(\eta^{\beta\alpha'}),|\mu|-|\nu|+|\beta|-|\gamma|,\xia) .\]
	
	Since $\nu=\beta\nu'$,
	\[ |\mu|-|\nu|+|\beta|-|\gamma| = |\mu|-|\beta\nu'|+|\beta|-|\gamma| = |\mu|-|\gamma\nu'|. \]
	
	From theorems \ref{thm.H-class.G-class.inverses} and \ref{thm.compositions.G.and.H},
	\dmal{(G_{(\mu)\alpha''}\circ H_{[\nu]\alpha''})(\eta^{\beta\alpha'})
		& = (G_{(\mu)\alpha''}\circ H_{[\nu]\alpha''} \circ G_{(\beta)\alpha'}\circ H_{[\gamma]\alpha'})(\xia) \\
		& = (G_{(\mu)\alpha''}\circ H_{[\nu']\alpha''} \circ H_{[\beta]\alpha'} \circ G_{(\beta)\alpha'}\circ H_{[\gamma]\alpha'})(\xia) \\
		& = (G_{(\mu)\alpha''}\circ H_{[\nu']\alpha''} \circ  H_{[\gamma]\alpha'})(\xia) \\
		& = (G_{(\mu)\alpha''}\circ H_{[\gamma\nu']\alpha''})(\xia).
	}
	
	It follows that
	\[ \Phi[st,\phi] = \Phi[s,\theta_t(\phi)]\Phi[t,\phi]. \]

\end{proof}

\section{Topological considerations}

Our goal in this section is to describe the topology on $\Gamma$ induced by the isomorphism given in Theorem \ref{thm:isomorphism.groupoid} from the topology of $\mathcal{G}_{tight}$ given in \cite{MR2419901}. First we recall a basis of compact-open sets for $\filt$ given in \cite{MR2974110}. For $e\in E(S)$ define
\[ U_e=\{\xi\in\filt\ |\ e\in\xi\} \]
and, if we are also given a finite (possibly empty) set $\{e_1,\ldots,e_n\}$, define
\[ U_{e:e_1,\ldots,e_n}= U_e\cap U_{e_1}^c\cap\cdots\cap U_{e_n}^c.\]

As observed in \cite{MR2974110},
\[ U_{e:e_1,\ldots,e_n}= U_e\cap U_{e_1e}^c\cap\cdots\cap U_{e_ne}^c\]
so that we may assume that $e_i\leq e$ for all $i\in\{1,\ldots,n\}$.

\begin{proposition}[\cite{MR2974110}, Lemmas 2.22 and 2.23]
	The sets $U_{e:e_1,\ldots,e_n}$ form a basis of compact-open sets for $\filt$ such that the resulting topology is Hausdorff.
\end{proposition}

This is the topology induced by the topology of pointwise convergence on characters, given the bijection between filters and characters.

\begin{remark}\label{remark:basis.tight}
	Since $\ftight$ is a closed subset of $\filt$, the sets
	\[  V_{e:e_1,\ldots,e_n} =  U_{e:e_1,\ldots,e_n}\cap\ftight\]
	as above form a basis of compact-open sets for the relative topology on $\ftight$. As a particular case, we denote $V_e=U_e\cap\ftight$. 
\end{remark}

Now we recall the topology on $\mathcal{G}_{tight}$ given in \cite{MR2419901}. Given $s\in S$ and an open set $\mathcal{U}\subseteq D_{s^*s}$, let
\[ \Theta(s,\mathcal{U})=\{[s,\phi]\in\mathcal{G}_{tight}\ |\ \phi\in\mathcal{U}\}. \]
Then, the collections of all $\Theta(s,\mathcal{U})$ is basis for a topology on $\mathcal{G}_{tight}$ making it a topological étale groupoid.

Finally we define a collection of subsets on $\Gamma$ as follows. Given $s=(\mu,A,\nu)\in S\lspace$ and $e_1,\ldots,e_n\in E(S)$, let
\[ Z_{s,e:e_1,\ldots,e_n}=\{(\eta^{\mu\gamma},|\mu|-|\nu|,\xi^{\nu\gamma})\in\Gamma\ |\ \xi\in V_{e:e_1,\ldots,e_n}\text{ and }H_{[\mu]\gamma}(\eta)=H_{[\nu]\gamma}(\xi) \}. \]

\begin{remark}
	In the above definition, if $e=s^*s$, we denote it by $Z_{s:e_1,\ldots,e_n}$. We also allow $n$ to be zero and in this case we denote it simply by $Z_{s,e}$ or $Z_s$ when $e=s^*s$.
\end{remark}

\begin{proposition}\label{prop:topology.groupoid}
	The sets $Z_{s,e:e_1,\ldots,e_n}$ defined above form a basis of compact-open sets for the topology on $\Gamma$ induced by the map $\Phi$ of Theorem \ref{thm:isomorphism.groupoid} from the topology on $\mathcal{G}_{tight}$ given by the sets $\Theta(s,\mathcal{U})$.
\end{proposition}

\begin{proof}
	Given an open set $\mathcal{U}\subseteq D_{s^*s}$ for some $s\in S\lspace$, the corresponding set of filters $V$ is an open set of $\ftight$. Let $[s,\phi]\in\Theta(s,\mathcal{U})$ with $\phi\in\mathcal{U}$ and associated filter $\xi\in V$ be given. Then there exists $e,e_1,\ldots,e_n$ such that $\xi\in V_{e:e_1,\ldots,e_n}\subseteq V$. Let $\mathcal{U}_{e:e_1,\ldots,e_n}$ be the corresponding open set on $D_{s^*s}$ so that $\phi\in\mathcal{U}_{e:e_1,\ldots,e_n}\subseteq\mathcal{U}$ and $[s,\phi]\in\Theta(s,\mathcal{U}_{e:e_1,\ldots,e_n})\subseteq\Theta(s,\mathcal{U})$.
	
	It follows that the family of all sets $\Theta(s,\mathcal{U}_{e:e_1,\ldots,e_n})$ is also a basis for $\mathcal{G}_{tight}$ where $s\in S\lspace$ and $\mathcal{U}_{e:e_1,\ldots,e_n}$ is the corresponding set of characters associated with the elements of $V_{e:e_1,\ldots,e_n}$.
	
	By Theorem \ref{thm.H-class.G-class.inverses}, $H_{[\mu]\gamma}(\eta)=H_{[\nu]\gamma}(\xi)$ if and only if $\eta=(G_{(\mu)\gamma} \circ H_{[\nu]\gamma})(\xi)$. By the definition of the isomorphism $\Phi$ of Theorem \ref{thm:isomorphism.groupoid}, we have that
	\[ \Phi(\Theta(s,\mathcal{U}_{e:e_1,\ldots,e_n}))=Z_{s,e:e_1,\ldots,e_n} \]
	and then it is easy to see that the sets $Z_{s,e:e_1,\ldots,e_n}$ form a basis for the induced topology by $\Phi$ on $\Gamma$.
	
	To see that they are compact, we use Proposition 4.15 of \cite{MR2419901}. We have the homeomorphisms
	\[ Z_{s,e:e_1,\ldots,e_n}\cong \Theta(s,\mathcal{U}_{e:e_1,\ldots,e_n})\cong \mathcal{U}_{e:e_1,\ldots,e_n} \cong V_{e:e_1,\ldots,e_n}, \]
	where the last is compact by Remark \ref{remark:basis.tight}.
\end{proof}

Let us now check that, with the above topology, $\Gamma$ is Hausdorff. This follows from an algebraic property of the semigroup as studied in \cite{MR2419901}.

\begin{definition}
	An inverse semigroup $S$ with zero is said to be $E^*$-unitary if for every $s\in S$ such that $e=se$ for some $e\in E(S)\setminus\{0\}$, we have that $s\in E(S)$.
\end{definition}

\begin{proposition}
	Let $\lspace$ be a weakly left-resolving labelled space, then $S\lspace$ is a $E^*$-unitary inverse semigroup.
\end{proposition}

\begin{proof}
	Let $s=(\mu,B,\nu)$ and $e=(\delta,D,\delta)\in E(S)\setminus\{0\}$. Suppose that $e=se$, that is
	\[ (\delta,D,\delta)=(\mu,B,\nu)(\delta,D,\delta). \]
	By the definition of the product on $S$, for this equality to be true it is necessary that $s\neq 0$, $\delta=\nu\delta'$ and $\delta=\mu\delta'$ for some $\delta'\in\awstar$. If that is the case then $\mu=\nu$, that is, $s=(\mu,B,\mu)\in E(S)$.
\end{proof}

\begin{corollary}
	$\Gamma$ with the above topology is Hausdorff.
\end{corollary}

\begin{proof}
	This follows immediately from Corollary 10.9 of \cite{MR2419901}.
\end{proof}	

%

%

Next, we show that $\Gamma$ can be seen as a Renault-Deaconu groupoid \cite{MR584266,MR1233967,MR1770333}. For that we define, for each $n\in\nn$ with $n\geq 1$, $\ftight^{(n)}=\{\xia\in\ftight\,|\,\alpha\in\lbf^{\leq \infty},\ |\alpha|\geq n\}$ and $\sigma:\ftight^{(1)}\to\ftight$ by $\sigma(\xia)=H_{[a]\gamma}(\xi)$ if $\alpha=a\gamma$.

\begin{proposition}
	Let $\ftight^{(1)}$ and $\sigma$ be as above. Then:
	\begin{enumerate}[(i)]
		\item $\ftight^{(1)}$ is open;
		\item $\sigma$ is a local homeomorphism.
	\end{enumerate}
\end{proposition}

\begin{proof}
	To prove (i), observe that if $\xi=\xia$ with $|\alpha|\geq 1$, then $(\alpha_1,r(\alpha_1),\alpha_1)\in\xi$. Hence
	\[\ftight^{(1)}=\bigcup_{a\in\alf}V_{(a,r(a),a)},\]	
	which is an open set.
	
	For (ii) we prove that, when restricted to $V_{(a,r(a),a)}$, $\sigma$ is a homeomorphism between $V_{(a,r(a),a)}$ and $V_{(\eword,r(a),\eword)}$. Denote by $\sigma_a$ the restriction of $\sigma$ to $V_{(a,r(a),a)}$, with $V_{(\eword,r(a),\eword)}$ as codomain. We use Theorem \ref{thm.H-class.G-class.inverses} in what follows. To see that $\sigma_a$ is injective, suppose that $\sigma_a(\xi)=\zeta^{\gamma}=\sigma_a(\eta)$. In this case, $\xi=\xi^{a\gamma}$, $\eta=\eta^{a\gamma}$ and $H_{[a]\gamma}(\xi)=\sigma_a(\xi)=\sigma_a(\eta)=H_{[a]\gamma}(\eta)$. Since $H_{[a]\gamma}$ is injective, $\xi=\eta$. For the surjectivity, let $\zeta=\zeta^{\gamma}\in V_{(\eword,r(a),\eword)}$ whence $\zeta\in \ftightw{(a)\gamma}$. Define $\xi=G_{(a)\gamma}(\zeta)$, so that $\xi\in V_{(a,r(a),a)}$ and $\sigma_a(\xi)=H_{[a]\gamma}(\xi)=\zeta$.
	
	To show that $\sigma_a$ is a homeomorphism, one can use Lemma \ref{lemma:A.belongs.h(F)} to conclude that
	\[\sigma_a(V_{(a\gamma,A,a\gamma):(a\gamma\eta_1,B_1,a\gamma\beta_1),\ldots,(a\gamma\eta_n,B_n,a\gamma\beta_n)})=V_{(\gamma,A,\gamma):(\gamma\eta_1,B_1,\gamma\beta_1),\ldots,(\gamma\eta_n,B_n,\gamma\beta_n)}\]
	for an arbitrary basic open set $V_{(a\gamma,A,a\gamma):(a\gamma\eta_1,B_1,a\gamma\beta_1),\ldots,(a\gamma\eta_n,B_n,a\gamma\beta_n)}$ of $V_{(a,r(a),a)}$, and Lemma \ref{lemma:A.belongs.g(F)} to see that
	\begin{multline*}
		\sigma_a^{-1}(V_{\eword,r(a),\eword}\cap V_{(\gamma,A,\gamma):(\gamma\eta_1,B_1,\gamma\beta_1),\ldots,(\gamma\eta_n,B_n,\gamma\beta_n)})\\
		=V_{(a\gamma,A\cap r(a\gamma),a\gamma):(a\gamma\eta_1,B_1\cap r(a\gamma\eta_1),a\gamma\beta_1),\ldots,(a\gamma\eta_n,B_n\cap r(a\gamma\eta_n),a\gamma\beta_n)}
	\end{multline*}
	for an arbitrary basic open set $V_{\eword,r(a),\eword}\cap V_{(\gamma,A,\gamma):(\gamma\eta_1,B_1,\gamma\beta_1),\ldots,(\gamma\eta_n,B_n,\gamma\beta_n)}$ of $V_{(\eword,r(a),\eword)}$. Hence $\sigma_a$ and $\sigma_a^{-1}$ are continuous.
\end{proof}

Notice that for $\xia\in\ftight^{(1)}$, the length of the word associated with $\sigma(\xi)$ is $|\alpha|-1$. This implies that for $n\in\mathbb{N}$ with $n\geq 1$, $\text{dom}(\sigma^n)=\ftight^{(n)}$. Also if $\xi=\xi^{\alpha\beta}$ then by Theorem \ref{thm.compositions.G.and.H}, $\sigma^{|\alpha|}(\xi)=H_{[\alpha]\beta}(\xi)$. This implies that
\[\Gamma=\{(\eta,m-n,\xi)\,|\,m,n\in\mathbb{N},\ \eta\in\text{dom}(\sigma^m),\ \xi\in\text{dom}(\sigma^n)\text{ and }\sigma^m(\eta)=\sigma^n(\xi)\},\]
that is, $\Gamma$ is the Renault-Deaconu groupoid associated with $\sigma$ \cite[Definition 2.5]{MR1770333}.

Now, let us prove that the topology given in Proposition \ref{prop:topology.groupoid} is the same as the one given in \cite{MR1770333}, which is the topology defined by the basic open sets
\begin{equation}\label{eq:open.set.Renault.Deaconu}
	\mathcal{V}(X,Y,m,n)=\{(\eta,m-n,\xi)\,|\,(\eta,\xi)\in X\times Y\text{ and }\sigma^m(\eta)=\sigma^n(\xi)\}
\end{equation}
where $X$ (resp. Y) is an open subset of $\text{dom}(\sigma^m)$ (resp. $\text{dom}(\sigma^n)$) for which $\sigma^m|_X$ (resp. $\sigma^n|_Y$) is injective.

\begin{lemma}\label{lemma:equality.open.sets}
	If $(\alpha,A,\beta)\in S\lspace$ and $e_1=(\beta\delta_1,B_1,\beta\delta_1),\ldots,e_n=(\beta\delta_n,B_n,\beta\delta_n)\in E(S)$ are such that $B_i\subseteq r(A,\delta_i)$, then $\sigma^{|\alpha|}$ (resp. $\sigma^{|\beta|}$) is injective restricted to $V_{ss^*:f_1,\ldots,f_n}$ (resp. $V_{s^*s:e_1,\ldots,e_n}$) and  $Z_{s:e_1,\ldots,e_n}=\mathcal{V}(V_{ss^*:f_1,\ldots,f_n},V_{s^*s:e_1,\ldots,e_n},|\alpha|,|\beta|)$, where $f_1=(\alpha\delta_1,B_1,\alpha\delta_1)$, $\ldots$,  $f_n=(\alpha\delta_n,B_n,\alpha\delta_n)$.
\end{lemma}

\begin{proof}
	For injectivity, it is sufficient to show that $\sigma^{|\alpha|}$ is injective in $V_{ss^*}$. This is indeed the case, for, if $\sigma^{|\alpha|}(\eta)=\sigma^{|\alpha|}(\xi)=\zeta^{\gamma}$ then $\eta=G_{(\alpha)\gamma}(\zeta)=\xi$ by Theorem \ref{thm.H-class.G-class.inverses}.
	
	Now, by the definition of $S\lspace$, $A\subseteq r(\alpha)\cap r(\beta)$. Since the labelled space is weakly left-resolving, this implies that for each $i=1,\ldots,n$, \[B_i\subseteq r(A,\delta_i)\subseteq r(r(\alpha)\cap r(\beta),\delta i)=r(\alpha\delta_i)\cap r(\beta\delta_i).\]
	From this, it follows that for $(\eta^{\alpha\gamma},|\alpha|-|\beta|,\xi^{\beta\gamma})\in\Gamma$, $\eta\in V_{ss^*:f_1,\ldots,f_n}$ if and only if $\xi \in V_{s^*s:e_1,\ldots,e_n}$.
	
	The equality between the sets in the statement is then a consequence of their respective definitions.
\end{proof}

\begin{proposition}
	The topologies on $\Gamma$ given by Proposition \ref{prop:topology.groupoid} and the basic open sets given by \ref{eq:open.set.Renault.Deaconu} coincide.
\end{proposition}

\begin{proof}
	By Lemma \ref{lemma:equality.open.sets}, it is sufficient to show that for each $(\eta,m-n,\xi)\in \mathcal{V}(X,Y,m,n)$, there exists $s\in S\lspace$ and $e_1,\ldots,e_n\in E(S)$ such that $(\eta,m-n,\xi)\in Z_{s:e_1,\ldots,e_n}\subseteq \mathcal{V}(X,Y,m,n)$. Given such $(\eta,m-n,\xi)$, there exist labelled paths $\alpha,\beta,\gamma$ such that $|\alpha|=m$, $|\beta|=n$, $\eta=\eta^{\alpha\gamma}$, $\xi=\xi^{\beta\gamma}$, and $H_{[\alpha]\gamma}(\eta)=H_{[\beta]\gamma}(\xi)$. Since $\xi\in Y$ and $Y$ is open there exists $e,e_1,\ldots,e_n\in E(S)$ such that $\xi\in V_{e:e_1,\ldots,e_n}\subseteq Y$. This implies, by Remark \ref{remark.when.in.xialpha}, that $e$ is of the form $(\beta\delta,A,\beta\delta)$ for some beginning $\delta$ of $\gamma$. Due to the construction of $H$, $A'=A\cap r(\alpha\delta)\neq \emptyset$. Define $s=(\alpha\delta,A',\beta\delta)$, so that $Z_{s:e_1,\ldots,e_n}$ is the desired set.
\end{proof}

\begin{corollary}
	The groupoid $\Gamma$ is amenable, so that the reduced and full C*-algebras coincide.
\end{corollary}

\begin{proof}
	This is an immediate consequence of Proposition 2.9 from \cite{MR1770333}.
\end{proof}

\section{The C*-algebra of the labelled graph as a groupoid C*-algebra}\label{section:cstar}

We begin this section recalling from \cite{MR584266} that, since $\Gamma$ is an étale groupoid, the corresponding Haar system is given by counting measures and, for $f,g \in C_c(\Gamma)$, their product is given by
\begin{equation}\label{eq:product.groupoid}
	(f*g)(\eta,n,\xi)=\sum_{\zeta,m:(\eta,m,\zeta)\in\Gamma}f(\eta,m,\zeta)g(\zeta,n-m,\xi).
\end{equation}

\begin{proposition}\label{prop:relations.in.grupoid.algebra}
	Consider the following functions in $C_c(\Gamma)$ viewed as elements of $C^*(\Gamma)$:
	\[ P_A = \chi_{Z_{(\eword,A,\eword)}} \]
	\[ S_a = \chi_{Z_{(a,r(a),\eword)}} \]
	where $A\in\acf$, $a\in \alf$ and $\chi$ represents the characteristic function of the given set. Then the families $\{P_A\}_{A\in\acf}$ and $\{S_a\}_{a\in\alf}$ satisfy the relations defining $C^*\lspace$.
\end{proposition}

\begin{proof}
	Fix an arbitrary element $(\eta,n,\xi)\in\Gamma$. By the definition of $Z_{(\eword,A,\eword)}$, $P_A(\eta,n,\xi)=1$ if and only if $n=0$, $\eta=\xi$, and $A\in\xi_0$. Similarly, by the definition of $Z_{(a,r(a),\eword)}$, $S_a(\eta,n,\xi)=1$ if and only if, $n=1$, $\eta=\eta^{a\gamma}$ and $H_{[a]\gamma}(\eta)=\xi$. For the remainder of the proof, let $A,B\in\acf$ and $a,b\in\alf$ be given. When we use (\ref{eq:product.groupoid}) bellow, all the sums are either zero or only have one term which is not zero, in which case the sum will be equal to $1$.
	
	Since $\xi_0$ is a filter in $\acf$, $\emptyset\notin\xi_0$ and hence $P_{\emptyset}=0$.
	
	Now,
	\[(P_A*P_B)(\eta,n,\xi)=\sum_{\zeta,m:(\eta,m,\zeta)\in\Gamma}P_A(\eta,m,\zeta)P_B(\zeta,n-m,\xi).\]
	This sum is not zero precisely when $m=n=0$, $\eta=\zeta=\xi$, $A\in\xi_0$ and $B\in\xi_0$. Using that $x_0$ is a filter in $\acf$, $A\in\xi_0$ and $B\in\xi_0$ if and only if $A\cap B\in\xi_0$, and therefore $P_A*P_B=P_{A\cap B}$.
	
	The equality $P_{A\cup B}=P_A+P_B-P_{A\cap B}$ follows from the fact that $\xi_0$ is an ultrafilter and therefore a prime filter, so that $A\cup B\in\xi_0$ if and only if $A\in\xi_0$ or $B\in\xi_0$.
	
	We compute both $P_A*S_a$ and $S_a*P_{r(A,a)}$ at $(\eta,n,\xi)$ and analyse when they are not zero. On the one hand
	\[(P_A*S_a)(\eta,n,\xi)=\sum_{\zeta,m:(\eta,m,\zeta)\in\Gamma}P_A(\eta,m,\zeta)S_a(\zeta,n-m,\xi),\]
	which is not zero if and only if $m=0$, $n=1$, $\zeta=\eta=\eta^{a\gamma}$, $A\in\eta_0$, $H_{[a]\gamma}(\eta)=\xi$ and $r(a)\in\xi_0$. On the other hand
	\[(S_a*P_{r(A,a)})(\eta,n,\xi)=\sum_{\zeta,m:(\eta,m,\zeta)\in\Gamma}S_a(\eta,m,\zeta)P_{r(A,a)}(\zeta,n-m,\xi),\]
	which is not zero if and only if $m=n=1$, $\eta=\eta^{a\gamma}$, $H_{[a]\gamma}(\eta)=\xi=\zeta$ and $r(A,a)\in\xi_0$. In the condition that $H_{[a]\gamma}(\eta)=\xi$ we have that $\xi_0\subseteq \eta_1$ and, in this case, $r(A,a)\in\xi_0$ if and only if $r(a)\in\xi_0$ (since $r(A,a)\subseteq r(a)$ and $\xi_0$ is a filter) and $A\in \eta_0$ (see the definition of complete family in Subsection \ref{subsection:filters.E(S)}). The equality $P_A*S_a=S_a*P_{r(A,a)}$ follows.
	
	We have that
	\begin{align*}
		(S_a^**S_b)(\eta,n,\xi)&=\sum_{\zeta,m:(\eta,m,\zeta)\in\Gamma}S_a^*(\eta,m,\zeta)S_b(\zeta,n-m,\xi)\\
		&=\sum_{\zeta,m:(\eta,m,\zeta)\in\Gamma}S_a(\zeta,-m,\eta)S_b(\zeta,n-m,\xi).
	\end{align*}
	A necessary condition for the above sum not to be zero is that the labelled path associated with $\zeta$ must start with $a$ and $b$ and, in particular, it is necessary that $a=b$. Hence, $S_a^**S_b=0$ if $a\neq b$. When $a=b$, the above sum is not zero if and only if $\zeta=\zeta^{a\gamma}$, $\eta=H_{[a]\gamma}(\zeta)=\xi$, $m=-1$, $n=0$ and $r(a)\in\xi_0$. This implies that $S_a^**S_a=P_{r(a)}$.
	
	Finally, for the last relation, let $A\in\acf$ be such that $0<\card{\lbf(A\dgraph^1)}<\infty$, and such that there is no $C\in\acf$ with $\emptyset\neq C\subseteq A\cap \dgraph^0_{sink}$. We need to verify that
	\begin{equation}\label{eq:last.relation}
		P_A=\sum_{a\in\lbf(A\dgraph^1)}S_a*P_{r(A,a)}*S_a^*=P_A*\sum_{a\in\lbf(A\dgraph^1)}S_a*S_a^*.
	\end{equation}
	Let us first apply $S_a*S_a^*$ to $(\eta,n,\xi)$,
	\begin{align*}
		(S_a*S_a^*)(\eta,n,\xi)&=\sum_{\zeta,m:(\eta,m,\zeta)\in\Gamma}S_a(\eta,m,\zeta)S_a^*(\zeta,n-m,\xi)\\
		&=\sum_{\zeta,m:(\eta,m,\zeta)\in\Gamma}S_a(\eta,m,\zeta)S_a(\xi,m-n,\zeta),
	\end{align*}
	which is not zero if and only if $m=1$, $n=0$, $\zeta=\zeta^{\gamma}$, $\xi=\xi^{a\gamma}$, $\eta=\eta^{a\gamma}$ and $H_{[a]\gamma}(\eta)=\zeta=H_{[a]\gamma}(\xi)$. Now
	\[P_A*(S_a*S_a^*)(\eta,n,\xi)=\sum_{\zeta,m:(\eta,m,\zeta)\in\Gamma}P_A(\eta,m,\zeta)(S_a*S_a^*)(\zeta,n-m,\xi),\]
	and using the above calculations, we see that the sum is not zero if and only if $m=n=0$, $A\in\eta_0$, $\zeta=\eta=\eta^{a\gamma}$, $\xi=\xi^{a\gamma}$ and $H_{[a]\gamma}(\eta)=H_{[a]\gamma}(\xi)$. Applying $G_{(a)\gamma}$ to the last equality we conclude that in this case $\eta=\gamma$. On the other hand, as seen above, $P_A(\eta,n,\xi)\neq 0$ if and only if $n=0$, $A\in\eta_0$ and $\xi=\eta$. By Theorem \ref{thm.tight.filters.in.es} and the assumptions on $A$, if $\eta=\eta^{\gamma}$, then $A\in\eta_0$ implies that $\gamma\neq\eword$, and in this case $\gamma=a\gamma'$ for some $a$ such that $a\in\lbf(A\dgraph^1)$. Hence, equation (\ref{eq:last.relation}) holds.

\end{proof}

Using the universal property of $C^*\lspace$, there is a *-homomorphism
\begin{equation}\label{eq:homomorphism}
	\pi:C^*\lspace\to C^*(\Gamma)
\end{equation}
such that $\pi(p_A)=P_A$ for all $A\in\acf$ and $\pi(s_a)=S_a$ for all $a\in \alf$. Our next goal is to show that $\pi$ is an isomorphism. To prove that $\pi$ is injective we use the gauge-invariance uniqueness theorem.

\begin{theorem}[\cite{MR3614028}, Corollary 3.10]
	Let $\lspace$ be a weakly left-resolving normal labelled space.
	Let $\{p_A , s_a\,|\, A \in \acf , a \in \alf\}$ be the universal representation of $\lspace$ that generates $C^*\lspace$, let $\{q_A , t_a\,|\, A \in \acf , a \in \alf\}$ be a representation of $\lspace$ in a C*-algebra $\mathcal{X}$ and let $\varphi$ be the unique *-homomorphism from $C^*\lspace$ to $\mathcal{X}$ that maps each $p_A$ to $q_A$ and each $s_a$ to $t_a$. Then $\varphi$ is injective if and only if $q_A$ is non-zero whenever $A\neq\emptyset$,	and for each $z\in\mathbb{T}$ there exists a *-homomorphism $\rho_z:C^*(\{q_A , t_a\,|\, A \in \acf , a \in \alf\}) \to
	C^*(\{q_A , t_a\,|\, A \in \acf , a \in \alf\})$ such that $\rho_z ( q_A ) = q_A$ and $\rho_z ( t_a ) = zt_a$ for $A\in\acf$ and $a\in\alf$.
\end{theorem}

\begin{corollary}
	The homomorphism $\pi$ in (\ref{eq:homomorphism}) is injective.
\end{corollary}

\begin{proof}
	To find the action $\rho$ of $\mathbb{T}$ on $C^*(\Gamma)$, we consider the one-cocycle $c:\Gamma\to\mathbb{R}$ given by $c(\eta,k,\xi)=k$ analogous to what is done in \cite{MR1233967}.
\end{proof}

To prove that $\pi$ is onto we follow the ideas of \cite{MR1432596}, but first we need a few lemmas.

\begin{lemma}\label{lemma:intersection.cylinder.groupoid}
	Let $(\alpha,A,\beta),(\mu,B,\nu)\in S\lspace$, then
	\begin{equation}\label{eq:intersection.cylinders}
		Z_{(\alpha,A,\beta)}\cap Z_{(\mu,B,\nu)} = 
		\begin{cases}
			Z_{(\mu,r(A,\delta)\cap B,\nu)}&\text{if }\mu=\alpha\delta,\ \nu=\beta\delta\text{ and }r(A,\delta)\cap B\neq\emptyset, \\
			Z_{(\alpha,A\cap r(B,\delta),\beta)}&\text{if }\alpha=\mu\delta,\ \beta=\nu\delta\text{ and }A\cap r(B,\delta)\neq\emptyset, \\
			\emptyset&\text{otherwise.}
		\end{cases}
	\end{equation}
\end{lemma}

\begin{proof}
	Let $(\eta,k,\xi)\in\Gamma$ be given. From the definition of $Z_{(\alpha,A,\beta)}$, $(\eta,k,\xi)\in Z_{(\alpha,A,\beta)}$ if and only if $\eta=\eta^{\alpha\gamma}$, $\xi=\xi^{\beta\gamma}$, $(\beta,A,\beta)\in\xi$ and $H_{[\alpha]\gamma}(\eta)=H_{[\beta]\gamma}(\xi)$. Similarly, $(\eta,k,\xi)\in Z_{(\nu,A,\mu)}$  if and only if $\eta=\eta^{\mu\varepsilon}$, $\xi=\xi^{\nu\varepsilon}$, $(\nu,B,\nu)\in\xi$ and $H_{[\mu]\varepsilon}(\eta)=H_{[\nu]\varepsilon}(\xi)$.
	
	If $(\eta,k,\xi)\in Z_{(\alpha,A,\beta)}\cap Z_{(\mu,B,\nu)}$ then one of the first two conditions of (\ref{eq:intersection.cylinders}) must hold. To show the equality is true in the first two cases, simply use the definition of complete family given in Subsection \ref{subsection:filters.E(S)} to conclude that $A\in\xi_{|\alpha|}$ if and only if $r(A,\delta)\in\xi_{|\mu|}=\xi_{|\alpha|+|\delta}$.
\end{proof}

\begin{lemma}\label{lemma:union.cylinder.groupoid}
	Let $(\alpha_1,A_1,\beta_1),\ldots,(\alpha_n,A_n,\beta_n)$ be elements of $S\lspace$. Then, there exist $B_1,\ldots,B_n\in\acf\setminus\{\emptyset\}$ such that $B_i\subseteq A_i$ for all $i=1,\ldots,n$ and
	\[\bigsqcup_{i=1}^{n}Z_{(\alpha_i,B_i,\beta_i)}=\bigcup_{i=1}^{n}Z_{(\alpha_i,A_i,\beta_i)}\]
	where $\sqcup$ represents a disjoint union.
\end{lemma}

\begin{proof}
	By induction on $n$. We start with the case $n=2$. Due to Lemma \ref{lemma:intersection.cylinder.groupoid}, we can suppose without loss of generality that $(\alpha_1,A_1,\beta_1)=(\alpha,A,\beta)$ and $(\alpha_2,A_2,\beta_2)=(\alpha\delta,B,\beta\delta)$ in such a way that $r(A,\delta)\cap B\neq\emptyset$. We claim that
	\[Z_{(\alpha,A,\beta)}\sqcup Z_{(\alpha\delta,B\setminus r(A,\delta),\beta\delta)}=Z_{(\alpha,A,\beta)}\cup Z_{(\alpha\delta,B,\beta\delta)}.\]
	The union on the left hand side is indeed disjoint by Lemma \ref{lemma:intersection.cylinder.groupoid} and, by the definition of these basic open sets, it is contained in the right hand side. To prove the other inclusion, let $(\eta,n,\xi)\in Z_{(\alpha\delta,B,\beta\delta)}$ be given. Since $\xi_{|\beta\delta|}$ is an ultrafilter, either $r(A,\delta)\cap B$ or $B\setminus r(A,\delta)$ belongs to $\xi_{|\beta\delta|}$. In the first case $(\eta,n,\xi)\in Z_{(\alpha,A,\beta)}$, and in the second case $(\eta,n\xi)\in Z_{(\alpha\delta,B,\beta\delta)}$.
	
	Now suppose we are given a disjoint union $\bigsqcup_{i=1}^{n}Z_{(\alpha_i,A_i,\beta_i)}$ (by using the induction hypothesis) and a basic open set $Z_{\alpha,A,\beta}$. We use the case $n=2$ with $Z_{\alpha,A,\beta}$ and $Z_{\alpha_1,A_1,\beta_1}$ to find $C_1\subseteq A$ and $B_1\subseteq A_1$ such that $Z_{\alpha,A,\beta}\cup Z_{\alpha_1,A_1,\beta_1}=Z_{\alpha,C_1,\beta}\sqcup Z_{\alpha_1,B_1,\beta_1}$. Now, use the case $n=2$ with $Z_{\alpha,C_1,\beta}$ and $Z_{\alpha_2,A_2,\beta_2}$ and repeat the process to find sets $B_i$ and $C_i$ such that $B_i\subseteq A_i$, $C_i\subseteq C_{i-1}$ and $Z_{\alpha,C_{i-1},\beta}\cup Z_{\alpha_i,A_i,\beta_i}=Z_{\alpha,C_i,\beta}\sqcup Z_{\alpha_i,B_i,\beta_i}$. Defining $B=C_n$, it follows that
	\[Z_{(\alpha,A,\beta)}\cup\bigsqcup_{i=1}^{n}Z_{(\alpha_i,A_i,\beta_i)}=Z_{(\alpha,B,\beta)}\sqcup\bigsqcup_{i=1}^{n}Z_{(\alpha_i,B_i,\beta_i)}.\]
	
\end{proof}

As in the case of $C^*\lspace$, for $\alpha=\alpha_1\ldots\alpha_n\in\lbf^{\geq 1}$, we will denote by $S_{\alpha}$ the product $S_{\alpha_1}*\cdots*S_{\alpha_n}$ in $C^*(\Gamma)$.

\begin{lemma}\label{lemma:1.cylinder.groupoid}
	For all $(\alpha,A,\beta)\in S\lspace$, $S_{\alpha}*P_A*S_{\beta}^*=\chi_{Z_{(\alpha,A,\beta)}}$.
\end{lemma}

\begin{proof}
	Following arguments similar to the ones used in the proof of Proposition \ref{prop:relations.in.grupoid.algebra}, one can show that $(S_{\alpha}*P_A*S_{\beta}^*)(\eta,A,\xi)$ is non-zero if and only if $(\alpha,A,\alpha)\in \eta$, $(\beta,A,\beta)\in\xi$, $\eta=\eta^{\alpha\gamma}$, $\xi=\xi^{\beta\gamma}$ and $H_{[\alpha]\gamma}(\eta)=H_{[\beta]\gamma}(\xi)$, but this happens if and only if $(\eta,A,\xi)\in Z_{(\alpha,A,\beta)}$.
\end{proof}

\begin{proposition}
	The map $\pi$ in (\ref{eq:homomorphism}) is surjective.
\end{proposition}

\begin{proof}
	As mentioned above, we follow the ideas of \cite{MR1432596}. The argument goes as follows. By Proposition \ref{prop:closure.span}, it suffices to show that $\mathcal{S}:=\text{span}\{S_{\alpha}*P_A*S_{\beta}^*\,|\,(\alpha,A,\beta)\in S\lspace\}=\text{span}\{\chi_{Z_{(\alpha,A,\beta)}}\,|\,(\alpha,A,\beta)\in S\lspace\}$ is dense in $C^*(\Gamma)$, where the equality is due to Lemma \ref{lemma:1.cylinder.groupoid}.
	
	Since $C_C(\Gamma)$ is dense in $C^*(\Gamma)$ it is sufficient to approximate an element $f\in C_C(\Gamma)$. Using Lemma \ref{lemma:union.cylinder.groupoid}, and the fact that the family of compact-open sets $\{Z_{(\alpha,A,\beta)}\,|\,(\alpha,A,\beta)\in S\lspace\}$ covers $\Gamma$, we can suppose that $f\in C(Z_{(\alpha,A,\beta)})$ for some $(\alpha,B,\beta)$. The sup norm dominates the I-norm, which dominates the norm in $C^*(\Gamma)$ (see \cite{MR584266}). The result will follow from the Stone-Weierstrass theorem by proving that $\mathcal{S}\cap C(Z_{(\alpha,A,\beta)})$ separates points.
	
	Due to Lemma \ref{lemma:intersection.cylinder.groupoid}, elements of $\mathcal{S}\cap C(Z_{(\alpha,A,\beta)})$ must be a linear combination of elements of the form $\chi_{Z_{(\alpha\delta,B,\beta\delta)}}$ for some $\delta$ and $B$ such that $r(A,\delta)\cap B\neq \emptyset$. Now fix two elements $(\eta^{\alpha\delta},n,\xi^{\beta\delta})$, $(\zeta^{\alpha\gamma},n,\rho^{\beta\gamma})$ of $Z_{(\alpha,A,\beta)}$ (here $n=|\alpha|-|\beta|$, and $\gamma$ and $\delta$ may be infinite paths). If there is a labelled path $\varepsilon$ that is a beginning of $\gamma$ and not $\delta$ (or vice-versa), then we can consider the characteristic function of the basic open set $Z_{(\alpha\varepsilon,r(A,\varepsilon),\beta\varepsilon)}$ to separate the points. Otherwise, $\delta=\gamma$. If this labelled path is infinite, then $\xi$ and $\rho$ are ultrafilters in $E(S)$ (by Theorem \ref{thm.tight.filters.in.es}) that must be equal since they have a common element, namely $(\beta,A,\beta)$. If the labelled path is finite, then it is also the case that $\xi=\rho$ because $\xi_{|\beta\delta|}$ and $\rho_{|\beta\delta|}$ are ultrafilters in $\acfrg{B}_{\beta\delta}$ that contain the element $r(A,\delta)$. This implies that $\xi=\rho$ (by the definition of complete family given in Subsection \ref{subsection:filters.E(S)}). So, in both cases $\xi=\rho$, and since $H_{[\alpha]\delta}$ is injective, by the definition of the groupoid it follows that $\eta=\zeta$, and we do not have two different points to separate.
\end{proof}

Putting everything together we have the following.

\begin{theorem}
	Let $\lspace$ be a normal weakly left-resolving labelled space, and $\mathcal{G}_{tight}$ the groupoid given in \cite{MR2419901}. Then $C^*\lspace\cong C^*(\mathcal{G}_{tight})$.
\end{theorem}

\bibliographystyle{abbrv}
\bibliography{labelledgraphs_ref}

\end{document}